\documentclass[12pt]{article}
\usepackage[utf8]{inputenc}
\usepackage{picture}
\usepackage{cite}
\usepackage{amsthm}
\usepackage{amssymb}
\usepackage{amsmath}
\usepackage{geometry}
\usepackage{multicol}
\usepackage{scalerel}
\usepackage{tikz-cd}
\usepackage{mathrsfs}
\usepackage{pstricks}
\usepackage{enumitem}
\usepackage[colorlinks=false]{hyperref}
\usetikzlibrary{calc}
\usetikzlibrary{matrix}
\usepackage{graphicx}
\usepackage[english]{babel}
\usetikzlibrary{arrows}
\usetikzlibrary{patterns}
\usepackage{booktabs}
\usepackage{float}
\usepackage{multirow}
\usepackage{subfig}
\theoremstyle{definition}
\newtheorem{thm}{Theorem}[section]
\newtheorem{lem}[thm]{Lemma}

\newtheorem{prop}[thm]{Proposition}
\newtheorem{cor}[thm]{Corollary}

\newtheorem{obs}[thm]{Remark}
\newtheorem{defn}[thm]{Definition}
\newtheorem{eg}[thm]{Example}

\newcommand{\bbZ}{\mathbb{Z}}
\newcommand{\bbR}{\mathbb{R}}
\newcommand{\bbF}{\mathbb{F}}
\newcommand{\bbQ}{\mathbb{Q}}
\newcommand{\bbC}{\mathbb{C}}
\newcommand{\OO}{\mathcal{O}}
\newcommand{\Gal}{\mathrm{Gal}}
\newcommand{\GL}{\mathrm{GL}}
\newcommand{\GSp}{\mathrm{GSp}}
\newcommand{\Sp}{\mathrm{Sp}}

\newcommand{\Frob}{\mathrm{Frob}}
\newcommand{\Hom}{\mathrm{Hom}}
\newcommand{\End}{\mathrm{End}}
\newcommand{\p}{\mathfrak{p}}
\newcommand{\q}{\mathfrak{q}}
\newcommand{\Q}{\mathbb{Q}}

\setlist[enumerate]{itemsep=0.2pt,topsep=0.2pt}

\begin{document}

\title{Potential Automorphy of K3 Surfaces with Large Picard Rank}
\author{Chao Gu}
\date{}
\maketitle
\abstract
The first part of this paper studies $\GSp_4$-type abelian varieties and the corresponding compatible systems of $\GSp_4$ representations. Techniques in \cite{BCGP} are applied to show that one can prove the potential modularity of these abelian varieties and compatible systems under some conditions that guarantee a sufficient amount of good primes. Then, in the second part, we use the potential modularity theorems to prove that K3 surfaces over totally real field $F$ with Picard rank $\ge17$ are potentially modular. 
\tableofcontents
\newpage

\section{Backgrounds}
\subsection{Introduction and Notations}

Let $F$ be a totally real field and $X/F$ be an algebraic $K3$ surface of geometric Picard rank $\rho$. Then the singular cohomology $H=H^2(X,\bbQ)$ with its Hodge structure splits as a direct sum 
\[H=T\oplus NS(X)\]
consisting of the N\'eron-Severi group with rational coefficients and the transcendental part $T$ of rank $22- \rho$ (\cite{Mor}), and associated to it is a corresponding
compatible system of Galois representations
$$\rho_p: G_F \rightarrow \mathrm{GO}_{22- \rho}(\bbQ_p)$$
Let  $\zeta_X(s)$ denote the Hasse-Weil zeta function
of $X$, then there is an equality
$$\frac{1}{\zeta_X(s) } =  \zeta(s) \zeta(s-2) L(T,s-1) L(NS,s-1)$$
The Hasse-Weil conjecture states that $\zeta_X(s)$ should extend to a meromorphic function for all $s \in \bbC$ and satisfy a functional equation. It is known that this holds for the $L$-function of the Artin motive $NS(X)$ due to Artin-Braeur theorem\cite{brauer}. So the conjecture would follow from the conjecture that the motive $T$ (or the associated compatible system $\{\rho_p\}$) is potentially automorphic. That is, there is a finite extension $F'/F$ such that the restrictions
of the motive become automorphic. The main theorem of this paper proves that this is true when the geometric Picard rank is larger than or equal to 17. 
\begin{thm}\label{main} Suppose that the geometric Picard rank of $X$ is $\rho\ge 17$.
Then $T$ is potentially modular, and the Hasse-Weil conjecture 
holds for $X$.
\end{thm}

\begin{eg} Suppose that $A/\Q$ is an abelian variety. Then associated to $A$
is a $K3$ surface $X$ known as the Kummer surface of $A$; it is a resolution of $A/[-1]$,
and it has geometric Picard rank $\ge 17$ (more precisely, $16 + \dim NS(A)$). The main theorem in this case then follows
from the main theorem of \cite{BCGP}.
\end{eg}

One key point is that are many other examples of $X$
with Picard rank $\ge 17$ which are \emph{not} of this form. Such $X$ \emph{do} however
turn out to be associated with either abelian surfaces $A$ or \emph{fake} abelian surfaces $A$.
Moreover, the varieties $A$ in question need not be defined over $F$, but rather are defined
over some finite Galois extensions $E/F$ and such that $A$ is isogenous to $A^{\sigma}$ 
for any $\sigma \in \Gal(E/F)$. These are analogues  of the so-called $\Q$-curves
for $\GL_2$ considered by Ribet \cite{Ribet}. The next step is to prove that all such varieties
(with motivic descent data to $F$) are potentially modular.

In the rest of this chapter, we will introduce the formal definitions of compatible systems of Galois representations, especially the $\GSp_4$ representations which will be our main interest of study in this paper. 

In Chapter 2 we introduce the possible abelian varieties that give compatible systems of $\GSp_4$ representations. We call these abelian varieties of $\GSp_4$-type, and we say that a compatible system is \textit{geometric} if it comes from such an abelian variety. 

In Chapter 3 we discuss the large image property of geometric compatible systems of $\GSp_4$ representations and proved that under the assumption of the large image property there are an ample source of \textit{good} primes $l$ such that the abelian variety satisfies some ordinariness condition for at least one prime $\lambda$ above $l$. 

In Chapter 4 we apply the Serre-Tate theory to construct a desired lift of $\bar{A}/\mathbb{F}_l$ for some good prime $l$, and in Chapter 6 we first use the theorem of Moret-Bailly to construct an abelian variety $\mathcal{A}$ together with two primes $\p$ and $\q$ such that $\mathcal{A}[\p]$ coincides with that of the original abelian variety $A$, and $\mathcal{A}[\q]$ is isomorphic to an induced representation. This process requires smoothness and connectedness of the moduli space of the $\GSp_4$-type abelian varieties, which is resolved in Chapter 5 through a construction of Shimura varieties. Finally, we apply the modularity lifting theorems to show the potential modularity of the constructed abelian variety $\mathcal{A}$, and hence the original abelian variety $A$ and the compatible system $\mathcal{R}_A$. 

In Chapter 7, we apply the motivic lifting results from \cite{Patrikis} to construct a compatible system $\mathcal{R}$ of Galois representations from the transcendental motive $T$, and we show that potential modularity theorem proved in Chapter 6 applies to $\mathcal{R}$, hence proving the main result. 

Note that our arguments do not apply to general compatible families of dimension $5$ with Hodge-Tate weights $[-1,0,0,0,1]$, even those assumed to arise from geometry. Not only do we need purity, but we also use the fact that the Tate conjecture is known for $K3$ surfaces \cite{Andre}. More generally, we exploit the existence of the Kuga-Satake construction in order to find motivic lifts of the representations $\rho_p$ to $\GSp_4(\overline{\bbQ}_p)$. The required ingredients we use can all be found in \cite{Patrikis}, though we also prove that the representations $\rho_p$ are absolutely irreducible in general which requires more work but also strongly relies on the assumption of large Picard rank.

To prove potential modularity (\S 6), we naturally use the main modularity lifting theorem of \cite[\S7]{BCGP} as well as the methods of \cite[\S9]{BCGP}. One notable difference with \cite{BCGP} is that for general $\GSp_4$-type abelian varieties we need to assume the base field is $\Q$ rather than a totally real field, or at least the abelian variety should have \textit{descent data} to $\bbQ$. In particular, we do not prove potential modularity for all abelian varieties $A$ of $\GSp_4$-type, that is, abelian varieties $A$ such that $\End(A) \otimes \Q$ contains a totally real field $K$ with $\dim(A) = 2[K:\Q]$ (\S2). The reason that we are not able to do this is that we are unable to prove that such abelian varieties admit enough primes $p$ splitting completely in $F$ such that $A$ is "distinguished ordinary", that is, primes $p$ for which $A$ is ordinary at all $v\mid p$ and the unit crystalline Frobenius eigenvalues are distinct modulo $p$ (\S3). As a result, more assumptions, including the ordinariness assumption (Definition \ref{ord}), and the large image assumption (Definition \ref{large}), need to be added to apply the potential modularity theorems. 

This paper will mainly use the notations as in \cite{BCGP}: For a perfect field $K$, we let $\OO_K$ be the ring of integers, $\overline{K}$ denote an algebraic closure of $K$ and $G_K$ the absolute Galois group $\Gal(\overline{K}/K)$. For each prime $p$ not equal to the characteristic of $K$, we let $\epsilon_p$ denote the $p$-adic cyclotomic character and $\bar{\epsilon}_p$ its reduction modulo $p$, and will commonly omit the $p$ and write as $\epsilon$ and $\bar{\epsilon}$ for simplicity. 

Let $K/\bbQ$ be a finite extension. If $v$ is a finite place of $K$ we write $K_v$ for the corresponding local field and $K_{(v)}$ for its residue field, and $\Frob_v:=\Frob_{K_v}$ be the geometric Frobenius element in $G_{K_v}$. If $v$ is a real place of $K$, then we let $[c_v]$ denote the conjugacy class in $G_K$ consisting of complex conjugations
associated to $v$.

\subsection{Compatible System of $\GSp_4$ representations}

\begin{defn} Let $F$ be a number field. By a rank $n$ weakly compatible system of $l$-adic representations $\mathcal{R}$ of $G_F$ defined over $M$ we mean a 5-tuple
\[\big(M,S,Q_v(X),\{\rho_\lambda\},\{H_\tau\}\big)\] where
\begin{enumerate}
    \item [(1)] $M$ is a number field;
    \item [(2)] $S$ is a finite set of primes of $F$;
    \item [(3)] For each prime $v\notin S$ of $F$, $Q_v(X)$ is a monic degree $n$ polynomial in $M[X]$;
    \item [(4)] For each prime $\lambda\mid l$ of $M$, we have
    \[\rho_\lambda:G_F\to \GL_n(\overline{M}_\lambda)\]
    a continuous semi-simple representation such that \begin{itemize}
        \item if $v\notin S$ and $v\nmid l$ is a prime of $F$ then $\rho_\lambda$ is unramified at $v$ and $\rho_\lambda(\Frob_v)$ has characteristic polynomial $Q_v(X)$;
        \item if $v\mid l$, then $\rho_{\lambda}|_{G_{F_v}}$ is de Rham, and in the case $v\notin S$ crystalline; 
    \end{itemize}
    \item [(5)] For $\tau:F\hookrightarrow \overline{M}$, $H_\tau$ is a multiset of $n$ integers such that for any $\overline{M}\hookrightarrow \overline{M}_\lambda$ over $M$ we have $\mathrm{HT}_\tau(\rho_{\lambda})=H_\tau$.
\end{enumerate}

We will call $\mathcal{R}$ strictly compatible if for each finite place $v$ of $F$ there is a Weil–Deligne representation $\mathrm{WD}_v(\mathcal{R})$ of the Weil group $W_{F_v}$ over $\overline{M}$ such that
for each place $\lambda$ of $M$ and every $M$-linear embedding $\iota:\overline{M}\hookrightarrow \overline{M}_\lambda$ we have
$$\iota \mathrm{WD}_v(\mathcal{R}) \cong \mathrm{WD}(\rho_\lambda|_{G_{F_v}})^{F-ss}.$$

We say that the compatible system $\mathcal{R}$ is pure of weight $w\in \bbZ$ if for a density one set of primes $v$ of $F$ with residue character $p$, each root of $\alpha$ of $Q_v(X)$ in $\overline{M}$ satisfies $|\iota(\alpha)|^2=p^w$ for all embeddings $\iota:\overline{M}\hookrightarrow \bbC$.
\end{defn}

In particular, we say $\mathcal{R}$ is a (weakly) compatible system of $\GSp_4$ representation if $n=4$ and $\rho_\lambda:G_F\to\GSp_4(\overline{K}_\lambda)$, here $$\GSp_4(K)=\{g\in \GL_4(K):gJg^t=\nu(g)J\},$$ where $\nu(g):\GSp_4\to \mathbb{G}_m$ is the similitude character, and $$J=\left(\begin{matrix}0&0&0&1\\0&0&1&0\\0&-1&0&0\\-1&0&0&0\end{matrix}\right).$$

Compatible systems of representations naturally arise from abelian varieties, as we have the following result: 

\begin{prop}(Theorem 2.8.1, \cite{BCGP}) If $A$ is an abelian variety over a number field $F$, then for
each $0\le i\le 2 \,\mathrm{dim}X$, the $l$-adic cohomology groups $H^i(A_{\overline{F}},\overline{\bbQ}_l)$ form a strictly
compatible system which is pure of weight $i$ and which is defined over $\bbQ$.
\end{prop}

\begin{defn} If $A/F$ is an abelian variety, then we may write the Galois representation $\rho_{A,l}$ as $H^1(A_{\overline{F}}, \overline{\bbQ}_l)$ (which is also the dual of the Tate module $T_l A(\overline{F})$), and 
$\mathcal{R}_A$ to be the compatible system $\{\rho_{A,l}\}$. 
\end{defn}

The next lemma shows that if one representation $\rho$ comes from $H^1(A,\overline{\bbQ}_p)$ for one $p$, then we can extend it to a compatible system of representations that all comes from the same abelian variety. 

\begin{lem}
    \label{extend}
Let $\rho$ denote an absolutely irreducible representation of $G_{F}$.
Suppose that $\rho$ occurs inside $H^1(A,\overline{\bbQ}_p)$ for some abelian variety $A/F$. 
Then $\rho$ extends to a weakly compatible system of irreducible Galois representations 
 $(M,S,\{Q_{v}(X)\},\{\rho_{\lambda}\},H)$
 for some number field $M$
 such that $\rho_{\lambda}$ occurs inside $H^1(A,\overline{\bbQ}_l)$ for each prime number $l$ and all $\lambda \mid l$.
\end{lem}

\begin{proof}
In the proof we will use the $H^1$ case of the Tate conjecture, which is known by the work of Faltings \cite{faltings1983}. We may assume that $A/F$ is simple. 

Let $D = \mathrm{End}^0(A)= \mathrm{End}(A) \otimes_{\bbZ} \Q$, and let $M$ denote the center
of $D$. For any prime $l$, the group $H^1(A,\Q_l)$ is a module over $M \otimes \Q_l = \prod_{\lambda \mid l} M_{\lambda}$,
and we may correspondingly write $H^1(A,\Q_l) = \prod_{\lambda \mid l} W_{\lambda}$.
The Galois representations $\rho_{\lambda}$ already give rise to a weakly compatible family
of representations of dimension $2 \dim(A)/[M:\Q]$  with coefficients in $M$
(see \cite[Thm~2.1.2]{Ribet} and \cite[\S11.10]{Shimura}).

On the other hand, by the Tate conjecture, we know that $$D \otimes_{\Q} \Q_l = \End_{G_{F}} H^1(A,\Q_l),$$
and hence $D \otimes_{M} M_{\lambda} = \mathrm{End}_{G_{F}}(W_{\lambda})$. 
It follows that, after some certain extension of scalars, we may split $W_{\lambda}$ as $n$ copies of some (absolutely) irreducible representation $V_{\lambda}$
of dimension $2 \dim(A)/(n [M:\Q])$ with $n^2 = [D:M]$. Here notice that an extension of scalars is only necessary
for the finitely many primes in $M$ which $D/M$ is ramified.
Since the representations $V^{n}_{\lambda}$ form the compatible system associated to $W_{\lambda}$, it follows that the $V_{\lambda}$ themselves form a compatible system also with coefficients in some fixed finite extension of $M$. 
\end{proof}

There is another way to get compatible systems of $\GSp_4$ representations, which is to consider the induction of compatible systems of $\GL_2$ representations. Let $F/E$ be a quadratic extension and $r:\rho_F\to \GL(V)=\GL_2(K)$ is a 2-dimensional representation. Choose $\sigma$ to be an element in $G_E\setminus G_F$, and suppose it satisfies that $\det r=\det r^\sigma$. Then $\det r$ extends to two possible characters $\chi_1$ and $\chi_2$ on $G_E$. Let $$\rho=\mathrm{Ind}_{G_F}^{G_E}\,r:G_E\to \GL(V\oplus \sigma V)=\GL_4(K),$$
and notice that $\wedge^2\rho$ is the sum of $\chi_1$, $\chi_2$ and some 4-dimensional representation. Therefore under some suitable basis $\rho$ gives simplectic representations $\rho_i:G_E\to \GSp_4(K)$ with similitude character $\chi_i$ for $i=1$ or 2. 

With our choice of $J$, the image will land in the subgroup
\[\left(\begin{matrix}
*&0&0&*\\
0&*&*&0\\
0&*&*&0\\
*&0&0&*
\end{matrix}\right)\cap \GSp_4(K).\]

In this paper we are only interested in the case where $\det r=\det r^\sigma=\epsilon^{-1}$ is the inverse cyclotomic character of $G_F$, and we take $\mathrm{Ind}_{G_F}^{G_E} r$ to be the symplectic representation that has similitude character the inverse cyclotomic character of $G_E$. For example, if $C/F$ is an elliptic curve and $r$ be the representation that corresponds to the dual of the Tate module of $C$, then $\mathrm{Ind}_{G_F}^{G_E}\,r$ would correspond to the dual of the Tate module of the abelian variety $\mathrm{Res}_E^F\, C$, with the symplectic structure coming from the Weil pairing.

In order to derive potential automorphy results we need the following ordinary properties of the compatible system. 

\begin{defn}
For a set of primes $S$ of $\bbQ$, we say $S$ is of strongly positive density if for any finite extension $F/\bbQ$, the set of primes over $S$ in $F$ has positive density. 

A set of primes $S$ is called potentially density zero if there exists some finite extension $F/\bbQ$ such that the set of primes over $S$ in $F$ has zero density. 

For example, it is clear by Chebotarev density theorem that the set of primes congruent to $1$ modulo $m$ for some positive integer $m>1$ has strongly positive density, while the other coprime congruence classes of primes have positive but not strongly positive density, hence are potentially density zero. 
\end{defn}

\begin{defn}\label{ord}
For a prime $p$ in $\bbQ$ that totally splits in $F$ and $M$, for a prime $\lambda\mid p$ in $M$, we say that $\mathcal{R}$ is $\lambda$-ordinary if the representations $\rho_\lambda|_{G_{F_v}}$ are ordinary mod $\lambda$ for all $v\mid p$ in $F$, and $p$-ordinary if it is $\lambda$-ordinary for all $\lambda\mid p$ in $M$. %which is equivalent to $v\nmid b_\lambda$ for all primes $v\mid p$ in $M$ and $\lambda\mid p$ in $F$. 

We say $\mathcal{R}$ is strongly positive ordinary if it is $p$-ordinary for a strongly positive set of primes in $\bbQ$. 

Given $\lambda$-ordinariness, for a prime $\lambda\mid p$ in $M$, $\mathcal{R}$ is called $\lambda$-distinguished ordinary if for any $v \mid p$ in $F$, $Q_{v}(X)$ has distinct unit roots mod $\lambda$. %It is equivalent to $v\nmid a_v^2-4b_v$ for all $v\mid p$. 
\end{defn}

\newpage

\section{Abelian Varieties of $\GSp_4$-type}

In this chapter we introduce a list of abelian varieties that give compatible systems of $\GSp_4$ representations. Let $F$ be a totally real Galois extension of $\bbQ$ and let $A$ be a simple principally polarized abelian variety over $F$ of dimension $g$ with no CM. We write $\End^0(A)=\End(A)\otimes \bbQ$. 

\subsection{Real $\GSp_4$-type Abelian Varieties}
\begin{defn}
We say that the abelian variety $A$ is of real $\GSp_4$-type if $K=\End^0(A)$ is a totally real field of degree $d$, with $g=2d$.
\end{defn}

Let $l$ be a prime in $\bbQ$ that totally splits in $F$ and suppose $l=\lambda_1\cdots \lambda_d$ also totally splits in $K$. Then $A[l]\cong\bigoplus_{i=1}^d A[\lambda_i]$ as  $\OO_K$-modules and thus we have the Tate modules
\[T_l(A)=\lim_{\longleftarrow}A[l^n]=\bigoplus_{i=1}^d T_{\lambda_i}(A)\] also decomposes into $d$ components. Also, the $l$-divisible groups $A[l^\infty]$, defined as the inductive system $(A[l^n],\iota_n)_n$, where the exact sequence holds: 
\[\begin{tikzcd}
0 \arrow[r] & {A[l^n]} \arrow[r, "\iota_n", hook] & {A[l^{n+1}]} \arrow[r, "{[l^n]}"] & {A[l^{n+1}]}
\end{tikzcd},\]
also decomposes into the "$\lambda_i$-divible" groups defined in the same way. That is, 
\[A[l^\infty]\cong\bigoplus_{i=1}^d A[\lambda_i^\infty].\]

For such primes $l$, let $\lambda\mid l$ be a prime over $l$ in $K$, then we can write $$\rho_{A,\lambda}^\vee:G_F\to \GL(T_\lambda A)$$ to be the Galois representation coming from the Tate module, and $$\bar{\rho}_{A,\lambda}^\vee:G_F\to \GL(A[\lambda])$$ be the residual representation that comes from $A[\lambda]$. Note that $\rho_{A,\lambda}^\vee$ and $\bar{\rho}_{A,\lambda}^\vee$ are the dual of $\rho_{A,\lambda}$ and $\bar{\rho}_{A,\lambda}$ respectively, the latter being the  Galois representation and the residual representation coming from $V_{A,\lambda}=H^1(A,K_\lambda)$.  

The following propositions are a series of well-known results of $l$-adic Galois representations that come from geometry: 

\begin{prop} When $A$ is of real $\GSp_4$-type, we have 
$$\rho_{A,\lambda}:G_F\to\GSp_4(K_\lambda)$$ and $$\bar{\rho}_{A,\lambda}:G_F\to \GSp_4(K_{(\lambda)})\cong \GSp_4(\bbF_l).$$ 

\end{prop} 

Let $p$ be a prime that totally splits in $F$ with $v\mid p$ a prime above $p$ in $F$, we have the following well-known property for representations $\rho_{A,\lambda}|_{G_{F_v}}$ from the theory of Weil polynomials.

\begin{prop}\label{twist}If the abelian variety $A$ is of $\GSp_4$-type, then the characteristic polynomial of Frobenius $\rho_{A,\lambda}(\Frob_v)$ is of form $$Q_v(X)=X^4-a_v X^3+b_v X^2-p\chi(\Frob_v) a_v X+p^2\chi^2(\Frob_v)$$  with coefficients in $\OO_K$, and is independent from the choice of $l$ and $\lambda$. Here $\chi$ is a totally even finite order character and we define $\zeta:=\chi(\Frob_v)$ to be a root of unity with uniformly bounded order. The roots of the polynomials are of form $$\alpha, p\zeta\alpha^{-1}, \beta, p\zeta\beta^{-1},$$ for some complex numbers $\alpha$ and $\beta$, and such that they are all Weil numbers, which are numbers $\pi$ such that for any embedding $\psi:\overline{\bbQ}\hookrightarrow \bbC$, $|\psi(\pi)|=\sqrt{p}.$ 
\end{prop}

By an argument in \cite{BCGP} Theorem 8.5.2, since $\chi$ is finite order and totally even we can find a totally real extension $F'/F$ such that $\chi|_{G_{F'}}\cong \psi^2$ and we can \emph{untwist} $\chi$ by considering the compatible system $\rho_{A,\lambda}|_{G_{F'}}\otimes (\psi^{-1})$. Therefore, from now on (in the following two cases) we may assume that $\chi$ is trivial and the compatible system of $\GSp_4$ representations will have similitude $\epsilon^{-1}$. 

\begin{prop}
Abelian variety $A$ of real $\GSp_4$-type gives a weakly compatible system of $\GSp_4$ representations of $G_F$ pure of weight 1, still denoted by $\mathcal{R}_A$, over $K$.  
\end{prop}

\subsection{Fake $\GSp_4$-type Abelian Varieties}
Now we introduce a type of abelian varieties that also give compatible systems of $\GSp_4$ representations of $G_F$. 

\begin{defn}
We say that the abelian variety $A$ is of fake $\GSp_4$-type if $D=\End^0(A)$ is a quarternion algebra over a totally real field $K$ of degree $d$, with $g=4d.$
\end{defn}

\begin{prop}
Abelian variety $A$ of fake $\GSp_4$-type gives a weakly compatible system $\mathcal{R}_A$ of $\GSp_4$ representations of $G_F$ defined over $K$ and is pure of weight 1. 
\end{prop}

\begin{proof}
For a prime $l$ in $\bbQ$ that totally splits in $K$ and $F$ and $\lambda\mid l$ in $K$ that is unramified in $D$, we have $D\otimes  K_\lambda\cong M_2(K_\lambda)$ and the characteristic polynomial of $\Frob_v$ for $v\mid p$ in $F$ would thus have form \[Q_{v}(X)=Q_{A,v}(X)^2,\] where 
\[Q_{A,v}(X)=X^4-a_v X^3+b_v X^2-p a_v X+p^2\]
with coefficients in $K$. Also, if we let  $\rho_{\lambda}:G_{F_v}\to\GL_8(K_\lambda)$ to be the the Galois representation that comes from $W_{A, \lambda}=H^1(A,K_\lambda)$ (which is also the dual of $\lambda$-adic Tate module $T_\lambda A$), then the image of $\rho_{\lambda}$ lies in $\big(\GSp_4(K_\lambda)\big)^2$. In other words, we can write $W_{A, \lambda}=V_{A,\lambda}\oplus V_{A,\lambda}$ as a direct sum of two identical $M_2(K_\lambda)$-modules and thus $\rho_{\lambda}=\rho_{A,\lambda}\oplus \rho_{A,\lambda}$ for some representation $\rho_{A,\lambda}:G_F\to\GSp_4(K_v)$, such that the characteristic polynomial of $\Frob_v$ with respect to $\rho_{A,\lambda}$ is $Q_{A,v}(X)$. 

For primes that are ramified in $D$, then after a quadratic extension $K'/K$ such that the quarternion algebra $D$ splits, or $D\otimes K'\cong M_2(K')$, we can still have $\rho_\lambda=\rho_{A,\lambda}\oplus\rho_{A,\lambda}$ 
for some $\rho_{A,\lambda}:G_F\to\GSp_4(K'_\lambda)$. Thus the characteristic polynomial of Frobenius would also have form
\[Q_v(X)=Q_{A,v}(X)^2,\] where 
\[Q_{A,v}(X)=X^4-a_v X^3+b_v X^2-pa_v X+p^2,\] and with coefficients of $Q_{A,v}$ also in the integer ring of $K$ (since the coefficients of $Q_v(X)$ are in $K$ and $Q_v(X)$ is monic). Thus the characteristic polynomial of $\Frob_v$ with respect to $\rho_{A,\lambda}$ is $Q_{A,v}(X)$. 

Now that we have definition of $\rho_{A,\lambda}$ on all $v$'s we are now able to show that $\rho_{A,\lambda}$'s form a compatible system of $\GSp_4$ representations. This comes from the fact that $\rho_{\lambda}$'s form a rank $8$ compatible system of representations and $\rho_{\lambda}\cong\rho_{A,\lambda}\oplus\rho_{A,\lambda}$. 
\end{proof}

\subsection{Abelian Varieties with Descent Data}
In \cite{Ribet}, Ribet introduced the $\bbQ$-curves, which are elliptic curves $A/F$ such that $A^\sigma$ is isogenous to $A$ for all $\sigma\in\Gal(F/\bbQ)$. He then proved that all $\bbQ$-curves are modular. Here we introduce an analogous definition for $\GSp_4$-type abelian varieties. 

\begin{defn}
Let $F/E/\bbQ$ be two Galois extensions of $\bbQ$. A (real or fake) $\GSp_4$-type abelian variety $A$ over $F$ is said to be with descent data to $E$ if $A^\sigma$ is $F$-isogenous to $A$ for all $\sigma\in \Gal (F/E)$. 
\end{defn}

\begin{prop}
If $A$ is with descent data to $E$, then $A$ gives a compatible system of twisted $\GSp_4$ representations of $G_E$. That is, there exists a compatible system $\mathcal{R}=(K,S,\{\rho_{\lambda}\},\{Q_v\},\{H_\tau\})$ of $\GSp_4$ representations $\rho_{\lambda}:G_E\to \GSp_4(\overline{K}_\lambda)$ such that  \[\rho_{\lambda}|_{G_F}=\chi\otimes \rho_{A,\lambda},\] where  $\chi:G_E\to \overline{K}_\lambda^\times$ a totally even finite order character and $\rho_{A,\lambda}$ is the $l$-adic representation of that comes from $V_{A,\lambda}\subseteq H^1(A,K_\lambda)$. The coefficients will not necessarily be in $K$, but will be in some fixed finite extension of $K$. 
\end{prop}

\begin{proof}
Since $A^\sigma$ is isogenous to $A$, we have isomorphisms denoted by $\psi_\sigma: V_{A^\sigma,\lambda}\to V_{A,\lambda}$ for $\sigma\in \Gal(F/E)$ (and thus well defined for each $\sigma\in G_E$) and we may first define the map $\tilde{\rho}_\lambda:G_E\to \GSp(V_{A,\lambda})$ as 
\[\tilde{\rho}_\lambda(\sigma)=\psi_\sigma\circ\sigma.\]
Note that the map isn't necessarily a homomorphism from $G_E$ to $\GSp(V_{A,\lambda})$, but the composition
\[\tilde{\rho}_\lambda(\sigma)\tilde{\rho}_\lambda(\tau)\tilde{\rho}_\lambda^{-1}(\sigma\tau)\] acts on $V_{A,\lambda}$ as $\psi_\sigma\psi_\tau^\sigma\psi_{\sigma\tau}^{-1}$, which can also be viewed as some $\overline{K}^\times$-valued function $c(\rho,\tau)$ on $G_E\times G_E$. It is easy to check that $c$ is a locally constant 2-cocycle on $G_E$ with value in $\overline{K}^\times$.

%that the representation $\rho_{A,\lambda}$ satisfies that
%\[\rho_{A,\lambda}^\sigma\cong\rho_{A,\lambda},\]
%where $\rho_{A,\lambda}^\sigma(g)=\rho_{A,\lambda}( g)$. 

%Therefore we can extend the $\rho_{A,\lambda}$ to $G_E$ (still denoted by $\rho_{A,\lambda}$) such that $\rho_{A,\lambda}^\sigma=\rho_{A,\lambda}\otimes \chi$ for some character $\chi:G_E\to \overline{K}_v^*$ such that $\ker\chi$ is $G_{E'}$ for some finite abelian extension $E'$ of $E$ and is contained in $F$. 

This means that there exists a well defined representation of \[\mathrm{P}\bar{\rho}_{\lambda}:G_E\to\mathrm{P}\GSp_4(K_\lambda).\] Now by a theorem of Tate \cite{Patrikis}, we have $H^2(G_E, \overline{K}^\times)$ being trivial and this means that we can find a lift $\rho_{\lambda}:G_E\to\GSp_4(\overline{K}_\lambda)$ such that its restriction to $G_F$ is a twist of the original representation $\rho_{A,\lambda}$ by a finite order character. 
%We can find a Galois extension $E'/E$ with $\Gal(E'/E)\cong (\bbZ/2\bbZ)^m$ such that the twists become trivial, thus restricted to $G_{E'}$ gives a compatible system with coefficients in $\bbQ$. 
In other words, we can write $\rho_{\lambda}|_{G_F}=\chi\otimes\rho_{A,\lambda}$ such that $\chi$ is a finite order character.

In addition, for a prime $v$ in $E$ that totally splits in $F$, since the Frobenii of primes above $v$ are $G_E$-conjugates of the other, we can deduce that after the twist $\chi$, the characteristic polynomial of $\Frob_{v_i}$ should be the same, thus the coefficients $a_{v_i}$ and $b_{v_i}$ are only dependent on $v$ and we may write the characteristic polynomials as 
$$Q_v(X)=X^4-a_v X^3+b_v X^2-pa_v X+p^2$$ (here $p$ is the residue characteristic) and will have coefficients in some fixed finite extension of $K$. 

\end{proof}

Now we may define the compatible system of $\GSp_4$ representations of $G_F$ over $M$ to be \textit{geometric}, if there exists an abelian variety $A/H$ over some Galois extension $H/F$  with descent data to $F$ such that either 
\begin{itemize}
    \item $A$ is of  real $\GSp_4$-type with $\End^0(A)= M$ and $H^1(A,M_\lambda)=V_{A,\lambda}$ gives the compatible system, or
    \item $A$  is of fake $\GSp_4$-type with $\End^0(A)= D$ for some quaternion algebra $D/M$ and $H^1(A,M_\lambda)=V_{A,\lambda}\oplus V_{A,\lambda}$ for sufficiently large prime $l$ and $\lambda\mid l$, and such that $V_{A,\lambda}$ gives the compatible system. 
\end{itemize} 

%Then geometric compatible system $\{\rho_v\}$ should have the property that the similitude character is $\nu(\rho_p)=\epsilon^{-1}\chi$ 

Now that we have formulated the correspondence between geometric compatible systems of $\GSp_4$ representations and $\GSp_4$-type abelian varieties, we can define distinguished ordinariness for abelian varieties as the following:

\begin{defn} Let $A$ be a $\GSp_4$-type abelian variety. For a prime $p$ in $\bbQ$ that totally splits in $F$ and $M$, let $\p$ be a prime above $p$ in $M$, we say that $A$ is $\p$-ordinary or $p$-ordinary if the corresponding $\GSp_4$ representations $\rho_{A,\p}$ are $\p$-ordinary or $p$-ordinary.  

We say $A$ is strongly positive ordinary if it is $p$-ordinary for a strongly positive set of primes in $\bbQ$. 

Given $\p$-ordinariness, for a prime $\p\mid p$ in $M$, we say that $A$ is $\p$-distinguished ordinary if for all $v\mid p$, the roots of $Q_{v}(X)$ that are units mod $\p$ (called unit roots) are distinct.
\end{defn}

An immediate proposition of the above definition of distinguished ordinariness is the following:
\begin{prop}\label{good}
     If $A$ is $p$-ordinary then the corresponding compatible system $\mathcal{R}_A$ is $p$-ordinary, and if the characteristic polynomial of $\Frob_v$ on $\rho_{A,\lambda}$ is denoted by
\[Q_v(x)=x^4-a_vx^3+b_vx^2-p a_vx+p^2,\] then $\p$-ordinariness is equivalent to $\p\nmid b_v$ for all primes $v\mid p$ in $F$. 

If $A$ is $\p$-distinguished ordinary, then further we have $\p\nmid a_v^2-4b_v$ for all $v\mid p$. 

\end{prop}
\begin{proof}
    Notice that $\p$-ordinariness follows from there being two unit roots mod $\p$, which is equivalent to $\p\nmid b_v$ by Newton polygon. 

    Given ordinariness, the two unit roots mod $\p$ satisfy the quadratic equation $x^2-a_vx+b_v=0$, and thus $\p$-distinguished-ness is equivalent to $\p\nmid a_v^2-4b_v$ for all $v\mid p$.  
\end{proof}
\newpage

\section{Large Image Property}
In this chapter we derive certain large image properties of the compatible system of Galois representations associated to abelian varieties of $\GSp_4$-type (with no CM), and show that we can find enough \textit{good} primes under certain large image hypotheses. 

\subsection{The Large Image Hypothesis}

We first derive an analogous result to \cite[Theorem 3.5]{Bar}. Notice that the arguments there don't apply directly in our case of real or fake $\GSp_4$-type abelian varieties since $\dim A/[K:\bbQ]=2$ or $4$ are even numbers.  
%\begin{lem} Suppose $A/F$ is of $\GSp_4$-type, then for sufficiently large prime $l$ that completely splits in $K$ and $F$, and $\lambda\mid l$ in $\OO_K$, the image of  $\bar{\rho}_{A,\lambda}:G_{F}\to \GSp_4(\bbF_l)$ at least contains $\Sp_4(\bbF_l)$. 
%\end{lem}
\begin{lem} (cf. \cite{Bar} Lemma 3.2) \label{largeLie}
Suppose $\mathcal{R}$ is a geometric compatible system of $\GSp_4$ representations of $G_F$ over $K$. Then for sufficiently large prime $l$ that completely splits in $F$ and $K$ and $\lambda\mid l$ in $K$, the semi-simple part of the Lie algebra of the $\lambda$-adic monodromy group is $\mathfrak{sp}_4$ or $\mathfrak{sl}_2\oplus \mathfrak{sl}_2$.
\end{lem}

\begin{proof} Throughout the following proof $l$ should be chosen as completely split in $K$ and $F$.
Recall that the $\lambda$-adic monodromy group is the $\lambda$-component of the $l$-adic monodromy group $G_l^{\mathrm{alg}}$, which is the algebraic closure of $\rho_{A,l}\otimes \bbQ_l$ in the group $\GL_{2g}/\bbQ_l$, and we write $(G_l^{\mathrm{alg}})'$ to be the derived subgroup.
%We have equalities of ranks of group schemes over $\bbQ_l$:
%\[\mathrm{rank}(G_l^{\mathrm{alg}})'=\mathrm{rank}\prod_{\lambda\mid l}\Sp_{4}/\bbQ_l.\]

If we write $\mathfrak{g}^{ss}=\mathrm{Lie}(G_l^\mathrm{alg})'$, and \[V_{A,l}=H^1(A,\bbQ_l)=\bigoplus\limits_{\lambda\mid l}V_{A,\lambda}\] 
where $V_{A,\lambda}=V_{A,l} \otimes K_\lambda$, and let $\bar{V}_{A,\lambda}=V_{A,\lambda}\otimes\overline{\bbQ}_l$. By the relation
\[G_l^\mathrm{alg}\subset\prod_{\lambda\mid l}\GSp(V_{A,\lambda})\cong\prod_{\lambda\mid l}\GSp_4/\bbQ_l\] we have
\[\mathfrak{g}^{ss}\subset\bigoplus\limits_{\lambda\mid l}\mathfrak{sp}_4(V_{A,\lambda})\]
%and we want to show that \[\mathfrak{g}^{ss}\otimes\overline{\bbQ}_l=\bigoplus\limits_{\lambda\mid l}\mathrm{sp}_4(\bar{V}_{A,\lambda}).\]

Projecting onto the $\lambda$-component we see that the image of $\mathfrak{g}^{ss}\otimes \overline{\bbQ}_l$ in $\mathfrak{sp}_4(\bar{V}_{A,\lambda})$ is
semisimple, and we can write it as a decomposition $$\bar{V}_{A,\lambda}=E(\omega_1)\otimes_{\overline{\bbQ}_l}\cdots\otimes_{\overline{\bbQ}_l}E(\omega_r), $$
where $E(\omega_i)$ for all $1\le i\le r$ are the irreducible (orthogonal or symplectic) Lie algebra modules of the highest weight $\omega_i$ corresponding to simple Lie algebras $\mathfrak{g}_i$
which are summands of the image 
\[\mathrm{Im}\big(\mathfrak{g}^{ss}\otimes\overline{\bbQ}_l\to\mathfrak{sp}_4(\bar{V}_{A,\lambda})\big)=\bigoplus_{i=1}^r\mathfrak{g}_i.\]

By [\cite{P}, Corollary 5.11] all simple factors of $\mathfrak{g}^{ss}\otimes\overline{\bbQ}_l$ are of classical type A, B, C or D and the weights $\omega_i$ are minimal. Since $\mathrm{dim}_{\overline{\bbQ}_l}\bar{V}_{A,\lambda}=4$, there are only two possible choices of decomposition: Either the product $E(\omega_1)\otimes_{\overline{\bbQ}_l}\cdots\otimes_{\overline{\bbQ}_l}E(\omega_r)$ consists of a single space $E(\omega_1)$ and the $\mathfrak{g}_1$-action on $E(\omega_1)$ is of type $C$ symplectic representation, thus $\bigoplus_{i=1}^r\mathfrak{g}_i=\mathfrak{sp}_4$, or it consists of two spaces $E(\omega_1)\otimes_{\overline{\bbQ}_l}E(\omega_2)$ each of dimension 2, with $\mathfrak{g}_i$ acts as $\mathfrak{sl}_2$ on $E(\omega_i)$, thus  $\bigoplus_{i=1}^r\mathfrak{g}_i=\mathfrak{sl}_2\oplus\mathfrak{sl}_2$.

%The rest of the theorem now follows according to the argument in Lemma 3.2 of \cite{Barnaszak}.

If $A$ is of fake $\GSp_4$-type, then for sufficiently large prime $l$, since the representation $W_{A,l}=H^1(A,\bbQ_l)$ can be decomposed into two identical components $W_{A,l}=V_{A,l}\oplus V_{A,l}$ and thus the original Galois representation $\rho_{l}$ can be decomposed as $\rho_{l}\cong\rho_{A,l}\oplus \rho_{A,l}$ for some $\GL_{2g}$-representation $\rho_{A,l}$, now the monodromy group $G_l^{\mathrm{alg}}$ is the algebraic closure of $\rho_{A,l}\otimes \bbQ_l$ and and $(G_l^{\mathrm{alg}})'$ be the derived subgroup. Also we have \[V_{A,\lambda}=V_{A,l}\otimes_K{K_\lambda} \] thus
\[V_{A,l}=\bigoplus_{\lambda\mid l}V_{A,\lambda}\]

By the exact same argument above on the semisimple part of the Lie algebra in the projection to $\mathfrak{sp}_4(V_{A,\lambda})$ we have the $\lambda$-adic monodromy group has semi-simple part of the Lie algebra $\mathfrak{sp}_4,$ or $\mathfrak{sl}_2\oplus\mathfrak{sl}_2$.

%\begin{lemm} (\cite{Barnaszak} Lemma 3.2 for fake $\GSp_4$-type abelian variety)
%We have equalities of ranks of group schemes over $\bbQ_l$:
%\[\mathrm{rank}(G_l^{\mathrm{alg}})'=\mathrm{rank}\prod_{\lambda\mid l}\Sp_{4}/\bbQ_l.\]
%\end{lemm}

\end{proof}

\begin{defn}\label{large}
We say that a geometric compatible system $\mathcal{R}$ of $\GSp_4$ representations (or the corresponding real or fake $\GSp_4$-type abelian variety) has large image property if for sufficiently large $l$ that splits completely in $F$ and $K$ and any $\lambda \mid l$ in $K$ we have the semi-simple part of the Lie algebra of the $\lambda$-adic monodromy group being $\mathfrak{sp}_4$. 
\end{defn}

\begin{prop} Given large image property, for sufficiently large $l$ that splits completely in $F$ and $K$ and $\lambda\mid l$ in $K$, the image of the residual representation $\bar{\rho}_{A,\lambda}$ at least contains $\Sp_4(\mathbb{F}_l)$.
\end{prop}
\begin{proof}
Note that Lemma \ref{largeLie} proved that the $h=2$ case in Lemma 3.2 in \cite{Bar} holds when $\{\rho_{A,\lambda}\}$ has large image property, and the same argument from Lemma 3.2-Lemma 3.5 implies that the $h=2$ case of Lemma 3.5 also holds for real or fake abelian varieties of $\GSp_4$-type that has large image property.

So if we consider the residual representation \[\bar{\rho}_{A,l}=\prod\limits_{\lambda\mid l}\bar{\rho}_{A,\lambda},\]
then for sufficiently large $l$ we have
\[\left[\bar{\rho}_{A,l}(G_F),\bar{\rho}_{A,l}(G_F)\right]=\prod\limits_{\lambda\mid l}\Sp_{4}(\bbF_l),\]
or
\[[\bar{\rho}_{A,\lambda}(G_F),\bar{\rho}_{A,\lambda}(G_F)]=\Sp_4(\bbF_l).\]

This means that for sufficiently large $l$ and $\lambda\mid l$, The residual representations $$\bar{\rho}_{A,\lambda}:G_F\to\GSp_4(\bbF_l)$$ will have image at least containing $\Sp_4(\bbF_l)$.
\end{proof}

We now define a residual representation to be $\textit{vast and tidy}$ in the sense of \cite{BCGP} Definition 7.5.6 and Definition 7.5.11. We also have the following lemma that shows the implication from large image property to vast and tidiness for $\GSp_4$ residual representations: 

\begin{lem} (cf. \cite{BCGP} Lemma 7.5.15)\label{vt} For $p>3$, if $\bar{\rho}: G_F\to \GSp_4(\bbF_p)$ that has image at least containing $\Sp_4$ and with similitude character
$\bar{\epsilon}^{-1}$, then $\bar{\rho}$ is vast and tidy.
\end{lem}

\begin{proof}
    For vastness notice that we still have the image of $\bar{\rho}|_{G_{F(\zeta_p^N)}}$ equal to $\Sp_4(\bbF_p)$ for all $N\in\bbZ^+$ (here $\zeta_p$ is the $p$-th root of unity). The rest follows from the arguments in \cite{BCGP} Lemma 7.5.15. 

    For $p>3$, tidiness follows from the fact that the center of the image of $\bar{\rho}$ (since the similitude character is $\bar{\epsilon}^{-1}$) will have order $p-1$ and \cite{BCGP} Lemma 7.5.12.
\end{proof}

Apart from this lemma we will also need the large image property for certain induced representation, which we will use in Chapter 7. Namely, we have:

\begin{lem} (\cite{BCGP} Lemma 7.5.22)\label{vtinduced} Suppose that $p>3$, ${F_1}/F$ is a quadratic extension such that $F_1$ is unramified at $p$, and $\bar{r}:G_{F_1}\to \GL_2(\bbF_p)$ restricted to $G_{F_1(\zeta_p)}$ has
image $\mathrm{SL}_2(\bbF_p)$. Choose $\sigma\in G_F\setminus G_{F_1}$, and assume that the determinants
$\mathrm{det}\, \bar{r}^\sigma = \mathrm{det}\, \bar{r}$ are equal to $\bar{\epsilon}^{-1}$ but the projective images $\mathrm{Proj}\,\bar {r}^\sigma\ncong\mathrm{Proj}\,\bar{r}$ are distinct. Let $\bar{\rho}:= \mathrm{Ind}^{G_F}_{G_{F_1}}\bar{r} : G_F\to \GSp_4(\bbF_p)$, then $\bar{\rho}$ is vast and tidy.

\end{lem}

\subsection{Distinguished Ordinary Primes}

\begin{lem} (cf. Lemma 4.2\cite{Bloch})\label{nolinear}
For a $\GSp_4$-type abelian variety $A/F$ that satisfies the large image property, given the characteristic polynomial of $\Frob_v$ for a prime $v$ with residue characteristic $p$,   \[Q_v(X)=X^4-a_v X^3+b_v X^2-p\zeta a_v X+p^2\zeta^2\] with coefficients in $K$ and $\zeta$ a root of unity with uniformly bounded order, then there is no linear relation between $p,\,a_v^2,\,b_v$ that can hold for a set of primes in $F$ with positive relative density in the set of primes whose characteristic splits completely in $F$ and $K$. In other words, each linear relation only holds on a potentially density zero set of primes. 
\end{lem}
\begin{proof}
Note that $p$, $a_v^2$ and $b_v$ are respectively the trace of $\Frob_v$ on $K_\lambda(1),\,V_{A,\lambda}\otimes V_{A,\lambda},$ and $\wedge^2 V_{A,\lambda}$. If there were such a relation, we would have a representation $W$ built out of copies of the above three spaces such that $\Frob_v$ has zero trace for a set of primes with positive relative density. 

From Lemma \ref{largeLie}, we know that for sufficiently large $l$ that splits completely and $\lambda\mid l$, the image of the residual representation $\bar{\rho}_{A,\lambda}$ contains at least $\Sp_4(K_{(\lambda)})$, so the image of $G_F$ on $\GL(V_{A,\lambda})$ is at least a subset of $\GSp_4(\OO_{K_\lambda})$ with positive measure. Thus a corresponding relation must hold for elements in $\GSp_4$ on a set of positive measure. Namely, if we write the eigenvalue of an element in $\GSp_4(\OO_{K_\lambda})$ as $\alpha,\,\beta,\,p\zeta\beta^{-1},\,p\zeta\alpha^{-1}$, then there would exist a linear relation between
\[p,\,(\alpha+\beta+p\zeta\alpha^{-1}+p\zeta\beta^{-1})^2,\,\text{ and }\]
\[2p\zeta+\alpha\beta+p\zeta\alpha^{-1}\beta+p\zeta\alpha\beta^{-1}+p^2\zeta^2\alpha^{-1}\beta^{-1}\]
that holds on a set with positive measure, hence for all $\alpha,\,\beta,\,p$, which is not possible. 

\end{proof}
\vskip 1 em

\begin{defn}
We define $p$ a good prime in $\bbQ$ if it satisfies the following properties:
\begin{enumerate}
    \item[(1)] $p$ totally splits in $K$ and $F$, and unramified in $D$ if applicable.

    \item[(2)] There exists a prime $\p\mid p$ of $K$ such that $A$ is $\p$-distinguished ordinary. 
    
    \item[(3)] For all $v\mid p$ in $F$, The residual representations $\bar{\rho}_{A,\p}:G_{F_v}\to\GSp_4(\bbF_p)$ at least contain $\Sp_4(\bbF_p)$. 

\end{enumerate}
\end{defn}
\vskip 1 em

\begin{lem}\label{suff}
Let $A/F$ be an abelian surface with $\End(A)=\bbZ$, or a fake abelian surface with $\End(A)$ an order in a quaternion algebra $D$ over $\bbQ$. If $A$ satisfies the large image property, then the good primes have strongly positive relative density in the set of primes that splits totally in $F$. 
\end{lem}

\begin{proof}
This is essentially Lemma 9.2.5 of \cite{BCGP}. 
    
For primes $p$ that splits totally in $F$ and $v\mid p$ in $F$, since each root $\pi$ of the characteristic polynomial satisfies that for any embedding $\psi:\overline{\bbQ}\hookrightarrow \bbC$, $|\psi(\pi)|=\sqrt{p}$, we have $|\psi(a_v)|\le 4\sqrt{p}$ and $|\psi(b_v)-2p|\le 4p$ for every $\psi$, thus $|\psi(a_v^2-4b_v)|\le 40p$ for every $\psi$. 

Now Lemma \ref{nolinear} shows that for each integer $c$ such that $|c|\le 6$, the relation $b_v=cp$ can only hold for a set of primes with potentially density zero. This means that the primes such that $p\nmid b_v$, or the ordinary primes, will have strongly positive relative density. 

Similarly, for each integer $|c|\le 40$, the relation $a_v^2-4b_v=cp$ can only hold for a set of primes with potentially density zero.

Combining with Proposition \ref{good} shows that the primes that are both ordinary and distinguished, hence good, have strongly positive relative density and this proves the lemma.  
\end{proof}

The previous lemma suggests that there are a sufficient amount of good primes for abelian varieties of $\GSp_4$-type when the coefficient field $K=\bbQ$. For general $K$, we have the following result: 

\begin{lem}\label{suff1}
Let $A/F$ be a strongly positive ordinary abelian variety of $\GSp_4$-type with descent data to $\bbQ$ and has large image property, then the good primes have positive relative density in the set of primes that splits totally in $K$ and $F$. 
\end{lem}
\begin{proof}
Lemma \ref{largeLie} shows that primes $p$ that satisfy (3) have relative density one, and thus the primes that satisfy both (3) and $p$-ordinariness have positive relative density. From now on we assume that $p$ splits completely in both $F$ and the coefficient field $K$. Again we use the fact from the previous lemma that $|\psi(a_v^2-4b_v)|\le 40p$ for every $\psi$. 

Recall that $\p$-distinguished ordinariness is equivalent to that $\p\nmid a_v^2-4b_v$ for all $v\mid p$. But since $A$ is with descent data to $\bbQ$, for $\sigma\in\Gal(F/\bbQ)$ we have $a_v^\sigma=\chi(\sigma)a_v$ and $b_v^\sigma=\chi(\sigma)^2b_v$. Thus for sufficiently large $p$, $\p\mid a_v^2-4b_v$ is equivalent to $\p\mid(a_v^\sigma)^2-4b_v^\sigma$. 

If $A$ is not $\p$-distinguished ordinary for any $\p$, then for each $\p$, $\p\mid a_v^2-4b_v$ for at least one $v\mid p$ thus for all $v\mid p$ due to the above deduction. Therefore for all $v\mid p$, $a_v^2-4b_v$ must be divisible by every prime above $p$, which means that $p\mid a_v^2-4b_v$. There are only finite number of choices for $\frac1p(a_v^2-4b_v)$ since they are algebraic integers in some fixed finite extension of $K$. And by Lemma \ref{nolinear}, for each choice such $v$ has potentially zero density. This means that primes satisfying (2) has positive relative density in the set of primes that totally splits in $K$.
\end{proof}

Suppose $A/F$ is a strongly positive ordinary abelian variety of $\GSp_4$-type with descent data to $\bbQ$, then we can choose $p$ and $q$ be two good primes and $\p$, $\q$ be two primes in $K$ above $p$ and $q$ respectively such that $A$ is both $\p$- and $\q$- distinguished ordinary. Furthermore, we can choose $\q$ to be such that it does not divide the discriminant of the characteristic polynomial of $\Frob_v$ for all $v\mid p$ in $F$. 

Thus $\Frob_v$ will have distinct unit roots, denoted by $\alpha$, $\beta$, and are also distinct mod $\q$. Thus the representation $\rho_{A,\q}|_{G_{F_v}}$ and the corresponding residual representation will be of form
$$\rho_{A,\q}|_{G_{F_v}}=\left(\begin{matrix}\lambda_{\alpha}&0&0&0\\0&\lambda_{\beta}&0&0\\0&0&\epsilon^{-1}\lambda_{\beta}^{-1}&0\\0&0&0&\epsilon^{-1}\lambda_{\alpha}^{-1}\end{matrix}\right),$$ 
 
$$\bar{\rho}_{A,\q}|_{G_{F_v}}=\left(\begin{matrix}\lambda_{\bar{\alpha}}&0&0&0\\0&\lambda_{\overline{\beta}}&0&0\\0&0&\bar{\epsilon}^{-1}\lambda_{\overline{\beta}}^{-1}&0\\0&0&0&\bar{\epsilon}^{-1}\lambda_{\bar{\alpha}}^{-1}\end{matrix}\right),$$
where $\lambda_\alpha$ ($\lambda_\beta$, etc.) is the character that sends $\Frob_v$ to $\alpha$ ($\beta$, etc.).

\newpage

\section{The Serre-Tate Canonical Lift}

Let $p$ be a prime that totally splits in $K$ and $\p\mid p$ be a prime above $p$ in $K$. In this chapter's first part, we derive the analogous results of Serre-Tate canonical lifts for $\p$-distinguished ordinary principally polarized abelian varieties $A$ over a finite field $k$ of characteristic $p$.

We first recall the original Serre-Tate theorem. 

\begin{thm} (Serre-Tate) Let $k$ be a finite field of characteristic $p>0$, and $R$ be a ring with a nilpotent ideal $I\subset R$ satisfying $R/I\cong k$. We define $\mathscr{A}(R)$ be the category of abelian schemes over $\mathrm{Spec}(R)$, and $\mathscr{D}(R)$ be the category of triples $(A_k, G,\mathcal{E})$, where \begin{itemize}
    \item $A_k$ is an abelian scheme over $\mathrm{Spec}(k)$; 
    \item $G$ is a $p$-divisible group over $\mathrm{Spec}(R)$;
    \item $\mathcal{E}$ is an isomorphism $A_k[p^\infty]\xrightarrow{\sim}G\times_R k$.
\end{itemize}

Then the functor $\Phi:\mathscr{A}\to\mathscr{D}$, where 
\[A\mapsto (A\times_R k, A[p^\infty], \text{natural } \mathcal{E}),\]
is an equivalence of categories. 
\end{thm}

\vskip 1 em

Now assume that the abelian variety $\bar{A}/k$ has good ordinary reduction at $\p$. Let $R$ be an artinian local ring and $\mathfrak{m}_R$ be the maximal ideal and $R/\mathfrak{m}_R=k$. Then there exists an integer $n$ such that $\mathfrak{m}_R^n=(0).$ By the original Serre-Tate theorem, the set of liftings of $A$ to $R$ is equivalent to the liftings of the $p$-divisible groups $$A[p^\infty]=\bigoplus\limits_{v\mid p}A[v^\infty]=A[\p^\infty]\oplus A[\bar{\p}^\infty],$$ 
where $$A[\bar{\p}^\infty]=\bigoplus\limits_{v\mid p,\,v\neq \p}A[v^\infty]. $$

We now want to consider the lift of $A[\p^\infty]$ to $\mathcal{A}[\p^\infty]$ in $R$. Since $A$ is $\p$-ordinary, we have a canonical product structure of $\p$-divisible group
\[A[\p^\infty]=\hat{A}\times T_{\p} A\otimes (K_\p/\OO_{K_\p}),\]
%\[0\longrightarrow\hat{A}\longrightarrow A[\p^\infty]\longrightarrow T_{\p} A\otimes (K_\p/\OO_{K_\p})\longrightarrow 0,\]
where $T_{\p} A\otimes (K_\p/\OO_{K_\p})$ is the \'etale quotient, and the toroidal formal group $\hat{A}$ satisfies \begin{align}\tag{4.1}\hat{A}\cong \Hom_{\OO_{K_\p}}(T_{\p}A^\vee,\hat{\mathbb{G}}_m).\end{align}

The identification of (4.1) is given by the following steps: \begin{enumerate}
        \item [(i)] We have the Weil pairings\[e_{\p^n}:A[\p^n]\times A^\vee[\p^n]\to \mu_{p^n}.\]
        \item [(ii)] Restricting the pairing to $\hat{A}$ gives the pairing \[e_{\p^n}:\hat{A}[\p^n]\times A^\vee[\p^n]\to \mu_{p^n}.\] This gives an isomorphism 
        \[\hat{A}[\p^n]\cong\Hom_{\OO_K}(A^\vee[\p^n],\mu_{p^n}).\]
        \item [(iii)] Taking inverse limit gives the isomorphism \[\hat{A}\cong\Hom_{\OO_{K_\p}}(T_\p A^\vee,\hat{\mathbb{G}}_m).\]
    \end{enumerate}

For a lift $\mathcal{A}/R$ of $A/k$, the $\p$-divisible part of $\mathcal{A}$ also has a canonical structure of extension
\[0\longrightarrow\hat{\mathcal{A}}\longrightarrow\mathcal{A}[\p^\infty]\longrightarrow T_\p A\otimes (K_\p/\OO_{K_\p})\longrightarrow 0\]
The isomorphism of $k$-groups (4.1) extends to isomorphism of $R$-groups 
\[\hat{\mathcal{A}}\cong \Hom_{\OO_{K_\p}}(T_\p A^\vee, \hat{\mathbb{G}}_m)\] via the pairing
\[E_\mathcal{A}:\hat{\mathcal{A}}\times T_\p A^\vee\to \hat{\mathbb{G}}_m\]
that canonically extends from the Weil pairing.

\vskip 1 em
\begin{lem}\label{canonicallift}
Let $p$ be a prime that totally splits in the number fields $K$ and $F$. If an abelian variety $A/F$ is $\p$-distinguished ordinary with an $\OO_K$ or $\OO_D$ (which is an order of a quaternion algebra over $K$) action, then for any $v\mid p$ in $F$ there exists a lift $\tilde{A}/F_v$ which is also $\p$-distinguished ordinary and with $\OO_K$ or $\OO_D$ action, and satisfies that \[\bar{\rho}_{\tilde{A},\p}|_{G_{F_v}}=\left(\begin{matrix}
\bar{\chi}_1&0&0&0\\0&\bar{\chi}_2&0&0\\
0&0&\bar{\epsilon}^{-1}\bar{\chi}_2^{-1}&0\\0&0&0&\bar{\epsilon}^{-1}\bar{\chi}_1^{-1}
\end{matrix}\right)=(\bar{\rho}_{A,\p}|_{G_{F_v}})^{ss},\]
a diagonal matrix.
\end{lem} 
\begin{proof}
Consider the reduction modulo $v$ of $A$, denoted by the $\p$-distinguished ordinary abelian variety $\bar{A}$ over the residue field $F_{(v)}$ of $F_v$ (which is isomorphic to $\bbF_p$). Let $W(F_{(v)})$ be the ring of Witt vectors over $F_{(v)}$, and take a series of Artinian local rings  $R_n=W(F_{(v)})/(v^n)$, and let $\mathcal{A}_n$ be the lift of $\bar{A}$ on $R_n$ such that: \begin{enumerate}
    \item [(i)] $\mathcal{A}_n[\bar{\p}^\infty]=A|_{W(F_{(v)})/(v^n)}[\bar{\p}^\infty]$; That is, $\mathcal{A}$ fixes everything else but the $\p$-divisible part in $A[p^\infty]$. 
    \item [(ii)] The lift $\mathcal{A}_n[\p^\infty]$ is the unique split lift of $\bar{A}[\p^\infty]$, that is, the  extension
    \[0\longrightarrow\hat{\mathcal{A}}_n\longrightarrow\mathcal{A}_n[\p^\infty]\longrightarrow T_\p A\otimes (K_\p/\OO_{K_\p})\longrightarrow 0\]
    is split.
\end{enumerate}

If we take the limit of the $R_n$'s we obtain a lift
\[\mathcal{A}=\lim_{\longrightarrow}\mathcal{A}_n\]
on $W(F_{(v)})$ that also has the canonical product structure $$\mathcal{A}=\hat{\mathcal{A}}\times T_\p A\otimes (K_\p/\OO_{K_\p}).$$ 

Since $W(F_{(v)})$ is isomorphic to $\OO_{F_v}$, we may as well think of $\mathcal{A}$ as defined over $F_v$. 

Now consider any endomorphism $f:A\to A$ that gives an endomorphism $\bar{f}$ on $\bar{A}$ hence an endomorphism on the $p$-divisible group $\bar{A}[p^\infty]$. 

For each $n$, since $\mathcal{A}_n$ does not change the $\bar{\p}$-divisible part, the action lifts to an endomorphism on $\mathcal{A}_n[\bar{\p}^\infty]$, and the split extension structure also canonically gives an endomorphism on the $\p$-divisible part $\mathcal{A}_n[\p^\infty]$ that lifts $\bar{f}|_{\bar{A}[\p^\infty]}$. This means that we can canonically lift $\bar{f}$ to an endomorphism on the whole $p$-divisible group, thus defining a morphism $$f_n:(\bar{A},G,\mathcal{E})\to (\bar{A},G,\mathcal{E})$$ on $\mathscr{D}(R_n)$. By Serre-Tate Theorem (considering the inverse functor of $\Phi$), this gives an endomorphism  $f_n\in \End(\mathcal{A}_n)$. Passing to limits we know that any endomorphism of $A$ gets deformed to an endomorphism on $\mathcal{A}$ thus we still have an $\OO_K$ or $\OO_D$ action on $\mathcal{A}$.  

In particular, we have the Frobenius morphism $\Frob_v$ on $\bar{A}$ lifts and since the action of it on the dual of $T_\p\bar{A}$ is of form $\Bigg(\begin{matrix}\bar{\chi}_1&0\\0&\bar{\chi}_2\end{matrix}\,\Bigg)$ due to $\p$-distinguished ordinariness, the Galois representation associated to the lift $\mathcal{A}$ would be of form $$\left.\bar{\rho}_{\mathcal{A},\p}\right|_{G_{F_v}}\cong\left(\begin{matrix}
\bar{\chi}_1&0&0&0\\0&\bar{\chi}_2&0&0\\
0&0&\bar{\epsilon}^{-1}\bar{\chi}_2^{-1}&0\\0&0&0&\bar{\epsilon}^{-1}\bar{\chi}_1^{-1}
\end{matrix}\right),$$ a diagonal matrix due to the product structure on $\mathcal{A}[\p^\infty]$. 

\end{proof}

\vskip 1 em
We call the $\tilde{A}$ we choose in Lemma \ref{canonicallift} the $\p$-\textit{canonical lift} of $\bar{A}$ to $F_v$. 

\vskip 1 em
\begin{cor}\label{diagonallift}
For a strongly positive ordinary abelian variety $A/F$ of $\GSp_4$-type with descent data to $\bbQ$ and satisfies the large image property, there exist good primes $p$ and $q$ of $A$ with $\p\mid p$ and $\q\mid q$ primes in $K$ such that $A$ is $\p$- and $\q$-distinguished ordinary, and for each $w\mid q$ in $F$ we can find a lift $\tilde{A}_w/F_w$ of $\bar{A}/F_{(w)}$ such that both $\bar{\rho}_{\tilde{A}_w,\q}|_{G_{F_v}}$ (for all $v\mid p$ in $F$) and $\bar{\rho}_{\tilde{A}_w,\q}|_{G_{F_w}}$ are diagonal matrices. 
\end{cor}
\begin{proof}
We choose the good primes $p$ and $q$ in $\bbQ$ with $\p\mid p$ and $\q\mid q$ in $K$ according to the arguments in the end of \S 3.2, and we know that $\bar{\rho}_{A,\q}|_{G_{F_v}}$ is of form $$\left(\begin{matrix}\lambda_{\bar{\alpha}}&0&0&0\\0&\lambda_{\overline{\beta}}&0&0\\0&0&\bar{\epsilon}^{-1}\lambda_{\overline{\beta}}^{-1}&0\\0&0&0&\bar{\epsilon}^{-1}\lambda_{\bar{\alpha}}^{-1}\end{matrix}\right),$$ a diagonal matrix. 

Now if we take $\tilde{A}_w/F_w$ to be the $\q$-canonical lift of $\bar{A}/F_{(w)}$, then it does not change the Galois representation on $G_{F_v}$. (in other words, it does not change the action of $\Frob_v$ acting on the $q$-divisible group $\tilde{A}_w[q^\infty]$.) Thus $\bar{\rho}_{\tilde{A}_w,\q}|_{G_{F_v}}=\bar{\rho}_{A,q}|_{G_{F_v}}$ would also be diagonal. 

On the other hand, since $\bar{A}/F_{(w)}$ is $\q$-distinguished ordinary with an $\OO_K$ or $\OO_D$ action, by Lemma \ref{canonicallift} the $\q$-canonical lift $\tilde{A}_w$ satisfies that \[\bar{\rho}_{\tilde{A}_w,\q}|_{G_{F_w}}=\big(\bar{\rho}_{A,\q}|_{G_{F_w}}\big)^{ss}=\left(\begin{matrix}
\bar{\chi}_1&0&0&0\\0&\bar{\chi}_2&0&0\\
0&0&\bar{\epsilon}^{-1}\bar{\chi}_2^{-1}&0\\0&0&0&\bar{\epsilon}^{-1}\bar{\chi}_1^{-1}
\end{matrix}\right),\]
is also a diagonal matrix.
\end{proof}
\vskip 1 em
\newpage

\section{Moduli Spaces of $\GSp_4$-type Abelian Varieties}
The goal of this chapter is to prove the following statement: 
\begin{thm} \label{MB}
Given a type (real or fake) and a totally real number field $K$ of degree $d$, and a quarternion algebra $B/K$ if the type is fake, suppose $\p,\q$ are two primes of $K$ above $p$ and $q$ respectively that split completely in $K$, then if we consider the moduli space $Y/F$ of all triples $(A,\iota_\p,\iota_\q)$ consisting of 
\begin{itemize}
    \item An abelian variety $A$ of dimension $g$ over $F$ of $\GSp_4$-type with $\End^0(A)=K$ (if real, so $g=2d$) or $B$ (if fake, so $g=4d$); %that is both $\p$-distinguished ordinary and $\q$-distinguished ordinary;
    \item A full level-$\p\q$ structure, in the sense of \begin{itemize}
        \item An isomorphism of group schemes $\iota_\p:A[\p]\to (\bbZ/p\bbZ)^g\times \mu_p^g$;
        \item An isomorphism of group schemes $\iota_\q:A[\q]\to (\bbZ/q\bbZ)^g\times \mu_q^g$;
    \end{itemize}    
\end{itemize}

then $Y$ is smooth and geometrically connected. 
\end{thm}

%We first recall the general setup on Shimura varieties of PEL-type: Let $B$ be a finite dimensional simple $\bbQ$-algebra that has center a totally real field $F$, and let $\OO_B$ be a maximal order of $B$. Let $*$ be a positive involution on $B$ that preserves $\OO_B$. Let $(V,\psi)$ be a skew-Hermitian $(B,*)$-module and we may define the algebraic group $G$ whose points in a $\bbQ$-algebra $R$ is given by

%$$G(R)=\{g\in \GL_B(V)\otimes_\bbQ R \mid  \psi(gx,gy)=\nu(g)\psi(x,y)\,\text{for some }\nu(g)\in R^\times\},$$

%We also fix a compact open subgroup $K=K^pK_p\subset G(\mathbb{A}^\infty)$ called the level structure, where $K^p\subset G(\mathbb{A}^{\infty,p})$ and $K_p=\prod_{v\mid p}K_{v}$. Here $\mathbb{A}^\infty$ denotes the ring of finite adeles of $\bbQ$ and $\mathbb{A}^{\infty,p}$ is the subring with trivial $p$-adic component.

%For a locally noetherian scheme $S$ We may consider the $S$-points on the moduli problem of isomorphism classes of quadruples $(A, \lambda,\iota,\bar{\eta})$, where
%\begin{itemize}
%    \item $A$ is a projective abelian scheme over $S$;
%    \item $\lambda:A\to A^\vee$ is a polarization of $A$;
%    \item $\iota:\OO_B\to \End(A)$, and
%    \item $\bar{\eta}=\eta^p\eta_p$ is the level structure, which is a $K$-class of $B\otimes \mathbb{A}^\infty$-linear isomorphisms 
%    \[\eta:V_{\mathbb{A}^\infty}(A)\to V(\mathbb{A}^\infty)\]
%    with the alternating forms on both sides up to a scalar multiple, where \[V_{\mathbb{A}^\infty}(A)=\prod_l V_l(A)\] is the Adelic Tate module of $A$. 
%\end{itemize}

\subsection{Construction of the Shimura Variety}

We will first construct the Shimura variety corresponding to $\GSp_4$-type abelian varieties following the arguments of \cite{kudla}. Now in either cases (real or fake) we let $B$ be an indefinite quaternion algebra over a totally real field $K$ with canonical involution $\iota$ and a maximal order $\OO_B$ such that $\OO_B^\iota=\OO_B$, and let $C=M_2(B)$ with the involution $x'=(x^\iota)^t$. Let $\{\sigma_1,\cdots ,\sigma_d\}$ be the set of real embeddings $K\hookrightarrow \bbR$. 

Let $V=\{x\in C:x'=x,\, \mathrm{tr}(x)=0\}$ be a five dimensional vector space over $K$ with a quadratic form $q$ defined by $x^2=q(x)\cdot \mathbf{1}_2$ of signature (3,2). Then the natural homomorphism $C(V)\to C$ extending the inclusion $V\hookrightarrow C$ to its Clifford algebra gives an isomorphism $C^+(V)\cong C$. We may let 
\[G=\mathrm{Res}_{K/\bbQ}\big(\mathrm{GSpin}(V)\big)\] be the restriction of scalars from $K$ to $\bbQ$ of the general Spin group \[\mathrm{GSpin}(V)=\{g\in C^\times: gg'=\nu(g)\}\] of $V$, so \[G(\bbQ)\cong \mathrm{GSpin}(3,2,K),\] and $G(\bbR)=G(\bbQ)\otimes\bbR$ acts on $V^d(\mathbb{R})$, which is $d$ copies of $V(\bbR)$, via the real embeddings $\sigma_i$ on the $i$-th component. We define $G(\bbQ)^+$ to be the connected component of elements in $G(\bbQ)$ such that $\nu(g)$ is a totally positive element in $K^\times$. 

Let $\mathcal{D}$ be the set of $z=(z_1,\cdots z_d)$ such that $z_i=(z_{i1},z_{i2})$ represents a negative $2$-plane in the $i$-th copy of $V(\bbR)$ in $V^d(\bbR)$ with $z_{i1}$ and $z_{i2}$ be a properly oriented basis such that the restriction of the quadratic form $q$ from $V(\bbR)$ to $z_i$ has matrix $-\mathbf{1}_2$ for the basis $z_{i1} , z_{i2}$. Then 
\[J_z:=(z_{i1}z_{i2})_{i=1}^d\in C^d(\mathbb{R})\] 
lies in $G(\mathbb{R})$ and defines a morphism on $\mathbb{R}$:
\[h_z:\mathbb{S}\to G,\quad h_z(i)=J_z\]
since $J_z'=-J_z$ and $J_zJ_z'=\mathbf{1}$. Thus $\mathcal{D}$ can be viewed as the space of conjugacy classes of such maps under the action of the group $G(\bbR)$ by
\[J_{gz}=gJ_zg^{-1}.\]

Choose an element $\tau\in B^\times$ such that $\tau^\iota=-\tau,\,\tau^2=-D(B)$ (where $D(B)$ is the discriminant of $B$) and $\tau\OO_B\tau^{-1}=\OO_B$, then the map $x\mapsto x^*=\tau x^{\iota}\tau^{-1}$ gives a positive involution on $\OO_B$. If we write $\OO_C=M_2(\OO_B)$ and $W=\OO_C^d$, viewed as left and right $C$ (thus, $G(\bbQ)$)-module, and let $\alpha=\mathrm{diag} (\tau,\tau)\in \OO_C$, then $\alpha'=-\alpha$ and the map
\[x\mapsto x^*:=\alpha x'\alpha^{-1}\] is a positive involution on $\OO_C$. 

We may define the alternating form \[\langle\cdot,\cdot\rangle: W\times W\to \bbQ\] by
\[\langle x,y\rangle=\sum_{i=1}^d\mathrm{tr}(y_i'\alpha^{-1} x_i).\]

Then for $c\in C$, we have
$$\langle cx,y\rangle=\sum_{i=1}^d\mathrm{tr}(y_i'\alpha^{-1}(cx_i))=\sum_{i=1}^d\mathrm{tr}(y_i'\alpha^{-1}c\alpha\alpha^{-1}x_i)=\langle x,c^*y\rangle$$

and $$\langle xc,y\rangle=\sum_{i=1}^d\mathrm{tr}(y_i'\alpha^{-1}x_ic)=\sum_{i=1}^d\mathrm{tr}(cy_i'\alpha^{-1}x_i)=\langle x,yc'\rangle.$$
In particular, \[\langle xg, yg\rangle=\nu(g)\langle x,y\rangle,\] for $g\in G$ and especially, 
\[\langle xJ_z,yJ_z\rangle=\langle x,y\rangle.\]

This means that $(W(R),\langle\cdot,\cdot\rangle)$ is a skew-Hermitian $(\OO_C(R),*)$-module with $G$ being the algebraic group of symplectic $K$-linear automorphisms  up to a similitude factor $\nu\in\mathrm{Res}_{K/\bbQ}\mathbb{G}_m\otimes_\bbQ R$. 

Now for a compact open subgroup $\kappa\subset G(\mathbb{A}^\infty)$, consider the functor $M_\kappa$ that associates to a locally noetherian scheme $S$ over $\bbQ$ the set of quadruples $(A, \iota, \lambda, \bar{\eta})$, such that 
\begin{itemize}
    \item $A$ is a abelian scheme over $S$ up to isogeny;
    
    \item $\iota:C\to \End^0(A)$ is a homomorphism that satisfies a determinant condition \[\mathrm{det}\big(\iota(c)\mid\mathrm{Lie}(A)\big)=N^2(c),\] where $N$ is the reduced norm on $C$. It also induces a homomorphism from $C$ to the endomorphism of the dual of $A$ denoted by $\iota^\vee:C\to\End^0(A^\vee)$.

    \item $\lambda$ is a $\bbQ$-class of polarizations on $A$ that is compatible to the action by $C$, that is, 
    \[\lambda\circ\iota(c)=\iota^\vee(c)\circ \lambda.\]

    \item $\bar{\eta}$ is a $\kappa$-class of $C\otimes \mathbb{A}^\infty$-linear isomorphisms 
    \[\eta:V_{\mathbb{A}^\infty}(A)\to W(\mathbb{A}^\infty)=W(\hat{\bbZ})\otimes\bbQ.\]
    that respect the symplectic forms on both sides up to a scalar multiple in $\mathbb{A}^\infty$. Here \[V_{\mathbb{A}^\infty}(A)=\prod_l T_l(A)\otimes \bbQ\cong H_1(A,\mathbb{A}^\infty)\] is the adelic Tate module of $A$.
\end{itemize}
Note that this abelian scheme $A$ will have dimension $8d$ over $S$. 

\begin{prop} (cf. \cite[Prop. 1.1]{kudla} )
For $\kappa$ neat (in the sense of \cite{Lan}) this moduli problem is representable by a smooth
quasi-projective scheme $M_\kappa$ over $\bbQ$ and \[M_\kappa(\bbC)\cong Sh_\kappa(G,\mathcal{D})(\mathbb{C}),\]
where $Sh_\kappa(G,\mathcal{D})$ is the Shimura variety defined by the PEL datum $(G,\mathcal{D})$. 
\end{prop}

\begin{proof} The representability is proved in \S 5 of \cite{kott} (our moduli problem is in case C so no further discussion is required), and we now show the correspondence of complex points following \cite{kudla}.

For each $i\in\{1,2,\cdots d\}$ and thus an inclusion $\sigma_i:K\hookrightarrow \bbR$, 
We can find $\tau_{i}=D(B)^{-1/2}\tau\in B^\times(\mathbb{R})$ so $\tau^2=-1$, and let  $\alpha_{i}=D(B)^{-1/2}\alpha.$ Choose an element $\beta_i\in B^\times(\bbR)$ such that $\beta_i\tau=-\tau\beta_i,\,\beta^\iota=-\beta$ and normalized such that $\beta^2=1$. 

Then $\Bigg(\begin{matrix}0&\beta_i\\-\beta_i&0\end{matrix}\Bigg)$ and $\Bigg(\begin{matrix} 0&\tau_i\beta_i\\-\tau_i\beta_i&0\end{matrix}\Bigg)$ forms a standard basis of a negative 2-plane on the $i$-th component of $V^d(\mathbb{R})$, denoted by $z_i$, and $$\left(\begin{matrix}0&\beta_i\\-\beta_i&0\end{matrix}\right)\left(\begin{matrix} 0&\tau_i\beta_i\\-\tau_i\beta_i&0\end{matrix}\right)=\left(\begin{matrix}\tau_i&0\\0&\tau_i\end{matrix}\right)=\alpha_i$$. 

So we write $z_0=(z_i)_{i=1}^d$ and $J_{z_0}=(\alpha_i)_{i=1}^d$. Then if $h=h_{z_0}$ be the map such that $h(i)=J_{z_0}$ then 
\[\langle xJ_{z_0},y\rangle=-\langle x,yJ_{z_0}\rangle=\langle yJ_{z_0},x\rangle,\]
and
\[\langle xJ_{z_0},x\rangle=\sum_{i=1}^d\mathrm{tr}\big(x_i'\alpha^{-1}x_i\alpha_i\big)=D(B)^{-1/2}\sum_{i=1}^d\mathrm{tr} (x_i^*x_i)>0.\]
Moreover, for $g\in G$,
\[\langle xJ_{gz_0},x\rangle=\langle xgJ_{z_0}g^{-1},x\rangle=\nu(g)^{-1}\langle xgJ_{z_0},xg\rangle.\]
This means that if we write $\mathcal{D}^+$ be the connected component of $\mathcal{D}$ containing $z_0$ then for $z\in \mathcal{D}^+$, $z=gz_0$ for some $g\in G^+$ and the pairing $\langle xJ_{z},y\rangle$ is symmetric and positive definite. 

Now for each $z\in \mathcal{D}^+$ we obtain a principally polarized abelian variety $A_z$ be the $8d$-dimensional complex structure $(W(\bbR),J_{z})$ modulo the lattice $W(\bbZ)$, and with principle polarization $\lambda$ induced by the pairing $\langle\cdot, \cdot\rangle$. It carries an action of $\OO_C$ by left multiplication (thus the homomorphism $\iota: C\to \End^0(A_z)$) and such that it is compatible with the polarization $\lambda$. Also, we have the isomorphism of the finite adelic Tate module
\[V_{\mathbb{A}^\infty}(A_z)\cong W(\hat{\bbZ})\otimes \bbQ\cong W(\mathbb{A}^\infty).\]

%Now for $\gamma\in \Gamma=\{g\in G(\bbQ)^+: W(\bbZ)g=W(\bbZ)\}$, right multiplication by $\gamma^{-1}$ gives an isomorphism between abelian varieties $A_z\to A_{\gamma z}$, thus $\Gamma\backslash \mathcal{D}^+$ parametrizes such principally polarized abelian varieties up to isomorphism. 

We can correspond the pair $(z,g\kappa)\in \mathcal{D}^+\times G(\mathbb{A}^\infty)/\kappa$ to the quadruple $(A_z, \iota, \lambda, \bar{\eta})$ where $\iota$ and $\lambda$ are defined as above and $\bar{\eta}$ be the $\kappa$-class of isomorphism induced by right multiplication by $g$ on $W(\mathbb{A}^\infty)$, that is, the composition
\[V_{\mathbb{A}^\infty}(A_z)\cong W(\mathbb{A}^\infty)\xrightarrow{\cdot g} W(\mathbb{A}^\infty).\] For $\gamma\in G(\bbQ)^+$, we can establish an isomorphism between quadruples corresponding to pairs $(z,g\kappa)$ and $(\gamma z,\gamma g\kappa)$ via the element in $\Hom(A_z,A_{\gamma z})\otimes_\bbZ \bbQ$ defined upon $W(\bbR)$ by right multiplication by $\gamma^{-1}$, and thus we have the correspondence
\[M_\kappa(\bbC)\cong G(\bbQ)^+\backslash\mathcal{D}^+\times G(\mathbb{A}^\infty)/\kappa\cong G(\bbQ)\backslash\mathcal{D}\times G(\mathbb{A}^\infty)/\kappa.\]

\end{proof}

%\begin{obs}
%\S 5 of \cite{kott} also showed representability for . For $p\nmid 2D(B)$, so $C(\bbQ_p)\cong$
%\end{obs}
\subsection{Proof of Theorem \ref{MB}}
Now we are able to describe the moduli problem in Theorem \ref{MB}. If $B$ is of fake type (an indefinite quaternion algebra that is division), then choose an idempotent element $e\in \OO_C$ and $eA$ gives an abelian variety of dimension $4d$ and has action by $\OO_B$. Conversely for a fake $\GSp_4$-type abelian variety $A$, $A\times A$ gives an abelian variety that has $\OO_C=M_2(\OO_B)$ action. 

On the other hand, if $B$ is of real type (so $B\cong M_2(K)$), then we may choose idempotent elements $e$ in $\OO_C$ and $e'$ in $\OO_B$ and $e'eA$ gives an abelian variety of dimension $2d$ and has action by $\OO_K$. Conversely for a real $\GSp_4$-type abelian variety $A$, $A\times A$ gives an ablian variety that has $\OO_B=M_2(\OO_K)$ action, and the product of four copies of $A$ gives an abelian variety that has $\OO_C$ action.

Therefore the moduli problems in Theorem \ref{MB} can be identified with the moduli problem described in \S 5.1 over $F$ with a certain full level-$\p\q$ structure. Smoothness easily follows from the construction of the Shimura variety, and in order to prove Theorem \ref{MB} we need to check connectedness.

Notice that the algebraic group $G\cong \mathrm{Res}_{K/\bbQ}\GSp_4$ is connected, we can deduce that $(G,\mathcal{D})$ is a PEL Shimura datum of type C (\cite{Milne}),  and connectedness now follows from the arguments in $\S 8$ of \cite{kott}, and the Hasse principle of $G$. 

\newpage
\section{Potential Modularity of $\GSp_4$-type Abelian Varieties}

\subsection{The Theorem of Moret-Bailly}

In this chapter, we first introduce the well-known theorem of Moret-Bailly, 

\begin{thm} (Proposition 3.1.1,\cite{blggt}) Let $K^{0}/K/E$ be number fields with $K^{0}/K$ and $K/E$ Galois. Suppose $S$ be a finite set of places of $E$ and $S^K$ be the set of places of $K$ above $S$. For $v\in S^K$, let $L_v'/K_v$ be a finite Galois extension with $L'_{\sigma v}=\sigma L'_v$ for $\sigma\in G_{E}$. Suppose also that $T/K$ is a smooth geometrically connected variety and that for each $v\in S^K$ we are given a non-empty, $\Gal(L_v'/K_v)$-invariant, open subset $\Omega_v\subset T(L_v')$. 

Then there exists a finite Galois extension $L/K$ and a point $P\in T(L)$ such that 
\begin{itemize}
    \item $L/E$ is Galois;
    \item $L/K$ is linearly disjoint from $K^{0}/K$;
    \item if $v\in S^K$ and $w$ a prime of $L$ above $v$ then $L_w/K_v$ is isomorphic to $L_v'/K_v$ and $P\in \Omega_v\subset T(L'_v)\cong T(L_w)$. 
\end{itemize}
\end{thm}

\vskip 1 em
The following proposition is a consequence of the theorem of Moret-Bailly, which is originally proved in \cite{cal12} and generalized in \cite{BCGP}. 
\begin{lem}\label{setup}
Let $G$ be a finite group, $E/\bbQ$ a finite Galois extension, and $S$ a set of primes of $E$. Let $E'/E,\,E^{0}/E$ be finite extensions, linearly disjoint from each other. 

Let $S'$ be the set of places in $E'$ over $S$. For each finite prime $v\in S'$, let $H'_v/E'_v$ be a finite Galois extension with a fixed inclusion $\phi_{v}:\Gal(H'_v/E'_v)\to G$ with image $D_v$. For each real prime $v\in S'$, let $c_v\in G$ be an element of order dividing 2.

Then there exists a number field $K/E$ and a finite Galois extension $L/K$ such that if $K'=KE'$ and $L'=LE'$, then
\begin{enumerate}
    \item[(1)] There is an isomorphism $\Gal(L/K)\to G$;
    \item[(2)] $L/E$ is linearly disjoint from $E^{0}E'/E$;
    \item[(3)] All places of $S'$ splits completely in $K'$;
    \item[(4)] For all local places $w$ above $v\in S'$, the local extension $L'_w/K'_w$ is equal to $H'_v/E'_v$, with a commutative diagram
    \[\begin{tikzcd}
    \Gal(L'_w/K'_w) \arrow[r] \arrow[d, Rightarrow, no head] & D_w \arrow[d, Rightarrow, no head] \\
    \Gal(H'_v/E'_v) \arrow[r] \arrow[r, "\phi_v"]            & D_v \end{tikzcd}
    \]
    \item[(5)] For all real places $w$ of $K'$ above $v\in S'$, complex conjugation $c_w\in G$ is conjugate to $c_v$.
\end{enumerate}
\end{lem}
\begin{proof} This is Proposition 9.1.12 of \cite{BCGP}.
\end{proof}
\vskip 1 em
We are going to use this lemma to show that in our construction in Lemma \ref{diagonallift}, the diagonal representations $\bar{\rho}_{\tilde{A}_w,\q}$ restricted to $G_{F_w}$ and all $G_{F_v}$'s are potentially induced representations.

\begin{lem} (cf. Lemma 9.2.7\cite{BCGP}) \label{Induced} Let $A/F$ be a strongly positive ordinary  $\GSp_4$-type abelian variety with descent data to $\bbQ$ and satisfying the large image property, and let $p, q$ be good primes chosen as in Lemma \ref{diagonallift}. Fix a real quadratic extension $F'/F$ such that $p$ and $q$ split completely and that it is linearly disjoint from the kernel of the action $G_F$ on $A[\p]$ and $A[\q]$. Then we can find a totally real Galois extension $F_1/F$ of $F$ and a quadratic extension $F_2/F_1:=F_1F'/F_1$, and a representation $\bar{r}:G_{F_2}\to \GL_2(\bbF_q)$, 
such that 

\begin{enumerate}
    \item [(1)]$p,\,q$ splits totally in $F_1$;
    \item [(2)]$F_1/F$ is linearly disjoint from $F'/F$ and the kernel of the action $G_\bbQ$ on $A[\p]$ and $A[\q]$; 
    \item [(3)]$\det \bar{r}=\bar{\epsilon}^{-1}$.
    \item [(4)]$\bar{r}(G_{F_2})=\GL_2(\bbF_q)$, and the projective image of $\bar{r}$ is not equal to its conjugate under $\Gal(F_2/F_1)$. 
    \item [(5)]Let $\bar{\rho}=\mathrm{Ind}_{G_{F_2}}^{G_{F_1}}\bar{r}$ with similitude character $\bar{\epsilon}^{-1}$, then it is vast and tidy, and satisfies\begin{enumerate}
        \item [(i)] For all primes $w\mid q$ of $F$ and $w_1\mid w$ of $F_1$, $$\bar{\rho}|_{G_{F_{1,w_1}}}\cong \bar\rho_{\tilde{A}_w,\q}|_{G_{F_w}}=\big(\bar\rho_{A,\q}|_{G_{F_w}}\big)^{ss},$$
        \item [(ii)] For all primes $v\mid p$ of $F$ and $v_1\mid v$ of $F_1$, 
        $$\bar{\rho}|_{G_{F_{1,v_1}}}\cong\bar\rho_{\tilde{A}_w,
        \q}|_{G_{F_v}}=\bar\rho_{A,\q}|_{G_{F_v}}. $$
    \end{enumerate}
\end{enumerate}

\end{lem}

\begin{proof}
Fix a prime $\ell\nmid pq$ in $F$. We are going to apply Lemma \ref{setup} with $G=\GL_2(\bbF_q)$, and with $E=F$, $E^0/F$ be the extension containing the kernels of $\bar{\rho}_{A,v}$ and $\bar{\rho}_{A,w}$ in $G_{F}$. Let $E'/F$ be any quadratic extension linearly disjoint from $E^0/F$. Let $S=\{\ell,p,q,\infty\}$. 
For infinite places $v$, we let $c_v$ to have eigenvalues 1 and $-1$. 

For prime $v\mid p$ in $F$, if we write $v=v_1v_2$ in $E'$, then we can let $\phi_{v_1}$ to be the representation corresponding to $\Bigg(\begin{matrix}\lambda_{\bar{\alpha}}&0\\0&\bar{\epsilon}^{-1}\lambda_{\bar{\alpha}}^{-1}\end{matrix}\Bigg)$ and $\phi_{v_2}$ corresponding to $\Bigg(\begin{matrix}\lambda_{\overline{\beta}}&0\\0&\bar{\epsilon}^{-1}\lambda_{\overline{\beta}}^{-1}\end{matrix}\Bigg)$. 
Similarly, we can write $w=w_1w_2$ in $E'$ for $w\mid q$ and let $\phi_{w_1}$ correspond to $\Bigg(\begin{matrix}\bar{\chi}_1&0\\0&\bar{\epsilon}^{-1}\bar{\chi}_1^{-1}\end{matrix}\Bigg)$ and $\phi_{w_2}$ correspond to $\Bigg(\begin{matrix}\bar{\chi}_2&0\\0&\bar{\epsilon}^{-1}\bar{\chi}_2^{-1}\end{matrix}\Bigg)$. 

Finally for $\ell=l_1l_2$ in $E'$, we let $\phi_{l_1}$ and $\phi_{l_2}$ to be such that they both have determinant $\bar{\epsilon}^{-1}$, $\phi_{l_1}$ be unramified while $\mathrm{Proj}\, \phi_{l_2}$ be ramified. 

Given the above setup, we can find a totally real extension $F_1$ (corresponding to the field $K$ as in Lemma \ref{setup}), and $F_2=F'F_1$ be the quadratic extension that satisfies (1) and (2), together with a representation $\bar{r}:G_{F_1}\to \GL_2(\bbF_q)$ that is vast and tidy due to Lemma \ref{vtinduced}. It satisfies (5) due to the construction in Lemma \ref{diagonallift}. (4) is satisfied because of the behavior of $\bar{r}$ on primes above $\ell$. By further replacing $F$ by a totally real quadratic extension (in which $p,q,\ell$ totally split) we can assume that the character $\bar{\epsilon}\det(\bar{r})$ has a square root in $F_1$. Replacing $\bar{r}$ by its twist by the square root of $\bar{\epsilon}\det(\bar{r})$ gives a representation that satisfies (3). 

\end{proof}

\vskip 1 em

\begin{thm} \label{trick}
Given a non-CM abelian variety $A/F$ of $\GSp_4$-type with descent data to $\bbQ$ that is strongly positive ordinary and has large image property. Let the totally real field $K$ be $\End^0(A)$ if $A$ is of real type and the base field of the quarternion algebra $D=\End^0(A)$ if $A$ is of fake type. Then we can find primes $p,q$ totally split in $F, K$ (and unramified in $D$ if $A$ is of fake type) and an abelian variety $\mathcal{A}/F'$, where $F'$ is a totally real extension of $F$ in which $p,q$ totally split, such that there exists $\p\mid p,\,\q\mid q$ in $K$ such that $\mathcal{A}[\p]\cong A[\p]$, and $\mathcal{A}[\q]$ is induced as representations. 
\end{thm}
\begin{proof}
We choose the good primes $p$ and $q$, together with $\p\mid p$ and $\q\mid q$ in $K$. For $w\mid q$ in $F$, we can find a $\q$-canonical lifting $\tilde{A}_w$ as in Lemma \ref{diagonallift}. Then from Lemma \ref{Induced} there exists a totally real extension $F_1/F$ such that  $\tilde{A}_w[\q]$ gives the dual of the induced representation $\bar{\rho}=\mathrm{Ind}_{G_{F_2}}^{G_{F_1}}\bar{r}$.

Now let $Y/F_1$ to be the moduli space of all triples $(B,\iota_\p,\iota_\q)$ consisting of $\GSp_4$-type abelian varieties $B$ with endomorphism ring either $\OO_K$ or $\OO_D$, together with symplectic isomorphisms
\begin{itemize}
    \item $\iota_\p:B[\p]\to A[\p]|_{G_{F_1}}$;
    \item $\iota_\q:B[\q]\to \bar{\rho}^\vee$ if $B$ is of real type, or $\iota_\q:B[\q]\to(\bar{\rho}^\vee)^2$ if $B$ is of fake type,
\end{itemize} 
as in \S5. For finite primes $\lambda$, let $\Omega_\lambda=Y^\mathrm{ord}(F_{1,\lambda})\subset Y(F_{1,\lambda})$ be the subspace of those that has good ordinary reduction. Then for every prime $\lambda\mid pq$, $\Omega_\lambda\neq\emptyset$, since we can choose $A$ itself for $\lambda\mid p$, and $\tilde{A}_w$ for $\lambda\mid w\mid q$ due to the construction in Lemma \ref{diagonallift}. For primes $\lambda\mid\infty, Y(F_{1,\lambda})$ is also nonempty because of the assumption $\det \bar{r}=\bar{\epsilon}^{-1}$.  

By the theorem of Moret-Bailly, we apply Theorem 6.1 wth $E=F, K=F_1$, $K^0$ be the compositum of $F_2$ and the kernels of $G_F$ on $A[\p]$, $\tilde{A}_w[\q]$, the extensions $L_v'/K_v$ be trivial, and $T/K=Y/F_1$ be the smooth geometrically connected variety, together with the open subsets $\Omega_v$'s.

Thus we can find a totally real field $F'/F$, linearly disjoint from $F_2/F_1$ and the kernels of $G_F$ on $A[\p]$, $\tilde{A}_w[\q]$, such that $p,q$ splits completely, and $Y(F')\cap\bigcap_{\lambda\mid pq}\Omega_\lambda$ is non-empty. This means that there exists an abelian variety $\mathcal{A}/F'$ such that $\mathcal{A}$ has good ordinary reduction mod every prime above $p,q$, and such that $\mathcal{A}[\p]\cong A[\p]$, and $\mathcal{A}[\q]$ is induced as Galois representation. 

\end{proof}

\begin{obs}\label{rem}
Note that when $K$ is $\bbQ$, which is the case where $A$ is an abelian surface or fake abelian surface, then we no longer need the descent data and the ordinariness assumption to find the good primes $p$ and $q$ due to Theorem \ref{suff}. 

In these cases, we have $\p=p$ and $\q=q$, so the $\q$-canonical lifts $\tilde{A}_w$ become the usual canonical lift of $\bar{A}/F_{(w)}$ to $F_w$, and the above Lemma \ref{Induced} becomes Lemma 9.2.7 in \cite{BCGP}. Theorem \ref{trick} holds due to the smoothess and connectedness on moduli spaces of abelian surfaces or fake abelian surfaces with full level-$pq$ structure. 
\end{obs}

\subsection{Potential Modularity Theorems}

First we recall the Galois representations associated to automorphic $\GSp_4$ representations. 

\begin{thm}(\cite{BCGP}, Theorem 2.7.1-2)
Suppose that $F$ is a totally real field and let $\pi$ be a cuspidal automorphic representation of $\GSp_4(\mathbb{A}_F)$ of parallel weight 2 and has central character $|\cdot|^2$. 

Fix a prime $p$ and a prime $\lambda\mid p$ in $K$, then there is a continuous semisimple representation $\rho_{\pi,\lambda}:G_F\to \GSp_4(\overline{K}_\lambda)$ satisfying the following properties:
\begin{enumerate}
    \item[(1)] The similitude character is $\nu(\rho_{\pi,\lambda})=\epsilon^{-1}$.
    \item[(2)] For each prime $v$ of $F$, we have
    \[\mathrm{WD}(\rho_{\pi,\lambda}|_{G_{F_v}})^{ss}\cong \mathrm{rec}_\lambda(\pi_v\otimes|\nu|^{-3/2})^{ss},\]
    where the $\mathrm{rec}_\lambda$ map that sends admissible irreducible complex representations of $\GSp_4(F_v)$ to a Frobenius-semisimple complex Weil–Deligne representations of the Weil group $W_{F_v}$ given by the local Langlands correspondence \cite{GT}. (Here we are fixing an embedding of coefficients $\overline{K}_\lambda\to \bbC$.)
    
    \item[(3)] If $v\mid p$, then $\rho_{\pi,\lambda}|_{G_{F_v}}$ is de Rham with Hodge-Tate weights $[0,0,1,1]$. Furthermore, if $\pi$ is $\lambda$-ordinary, then 
    \[\rho_{\pi,\lambda}|_{G_{F_v}}=\left(\begin{matrix}
        \lambda_{\alpha_v}&*&*&*\\
        0&\lambda_{\beta_v}&*&*\\
        0&0&\epsilon^{-1}\lambda_{\beta_v}^{-1}&*\\
        0&0&0&\epsilon^{-1}\lambda_{\alpha_v}^{-1}
    \end{matrix}.
    \right)\]
    \item[(4)] If $\rho_{\pi,\lambda}$ is irreducible, then for each finite prime $v$ of $F$, $\rho_{\pi,\lambda}|_{G_{F_v}}$ is pure.
\end{enumerate}
\end{thm}

\begin{defn}
Let $F$ be a totally real number field. We say that a representation $\rho:G_F\to\GSp_4(\overline{K}_\lambda)$ is modular (with similitude $\epsilon^{-1}$) if there is a cuspidal automorphic representation $\pi$ of $\GSp_4(\mathbb{A}_F)$ of parallel weight 2 and central character $|\cdot|^2$ that satisfies $\rho\cong\rho_{\pi,\lambda}$. We say that $\rho$ is potentially modular if there is a finite Galois extension $F'/F$ of totally real fields such that $\rho|_{G_{F'}}$ is modular. 

For a compatible system of $\GSp_4$ representations $\mathcal{R}$, we say that $\mathcal{R}$ is modular (potentially modular) if $\rho_\lambda$ is modular for some (equivalently, for any) $\lambda$.

If $A/F$ is an abelian variety of $\GSp_4$-type, then we say $A$ is modular (potentially modular) if $\mathcal{R}_A$ is modular, in other words, $\rho_{A,\lambda}$ is modular (potentially modular) for some (equivalently, for any) $\lambda$.
\end{defn}

\begin{thm} (\cite{BCGP}, Theorem 8.4.1)\label{modular} Suppose that $p\ge 3$ splits completely in $K$ and the totally real field $F/\bbQ$. Fix $\p\mid p$ a prime in $K$. Suppose that $\rho:G_F\to \GSp_4(\overline{K}_\p)$ satisfies:\begin{enumerate}
    \item [(1)] The similitude character of $\rho$ is $\epsilon^{-1}$;
    \item [(2)] The representation $\bar{\rho}$ is vast and tidy;
    \item [(3)] For all $v\mid p$, $\bar{\rho}|_{G_{F_v}}$ is conjugate to a representation of form
    \[\left(\begin{matrix}\bar{\lambda}_{\alpha_v}&0&*&*\\
    0&\bar{\lambda}_{\beta_v}&*&*\\
    0&0&\bar{\epsilon}^{-1}\bar{\lambda}_{\beta_v}^{-1}&0\\
    0&0&0&\bar{\epsilon}^{-1}\bar\lambda_{\alpha_v}^{-1}
    \end{matrix}\right)\]
    where $\alpha_v\neq\beta_v$;
    \item [(4)] There exists $\pi$ of parallel weight 2 and central character $|\cdot|^2$, which is ordinary at all $v\mid p$, such that $\bar{\rho}_{\pi,\p}\cong \bar{\rho}$.
    \item [(5)] For all finite places $\p$ of $F$, $\rho|_{G_{F_v}}$ and $\rho_{\pi,\p}|_{G_{F_v}}$ are pure. 
\end{enumerate}

Then $\rho$ is modular. More precisely, there exists an ordinary automorphic representation $\pi'$ of $\GSp_4(\mathbb{A}_F)$ of parallel weight 2 and central character $|\cdot|^2$ that satisfies $\rho_{\pi',\p}\cong \rho$. 
\end{thm}

For twisted Galois representations, we also have the following theorem.
\begin{thm} (\cite{BCGP}, Theorem 8.5.2)\label{modulartwist} Suppose that $p\ge 3$ splits completely in $K$ and the totally real field $F/\bbQ$. Fix $\p\mid p$ a prime in $K$. Suppose that $\rho:G_F\to \GSp_4(\overline{K}_\p)$ satisfies:\begin{enumerate}
    \item [(1)] The similitude character of $\rho$ is $\epsilon^{-1}\chi$, where $\chi$ is a totally even finite order character, and is unramified at all places above $p$;
    \item [(2)] The representation $\bar{\rho}$ is vast and tidy;
    \item [(3)] For all $v\mid p$, $\bar{\rho}|_{G_{F_v}}$ is conjugate to a representation of form
    \[\left(\begin{matrix}\bar{\lambda}_{\alpha_v}&0&*&*\\
    0&\bar{\lambda}_{\beta_v}&*&*\\
    0&0&\bar{\chi}\bar{\epsilon}^{-1}\bar{\lambda}_{\beta_v}^{-1}&0\\
    0&0&0&\bar{\chi}\bar{\epsilon}^{-1}\bar\lambda_{\alpha_v}^{-1}
    \end{matrix}\right)\]
    where $\alpha_v\neq\beta_v$;
    \item [(4)] There exists $\pi$ of parallel weight 2 which is an unramified twist of an ordinary representation at all $v\mid p$ (called twisted ordinary), such that $\bar{\rho}_{\pi,\p}\cong \bar{\rho}$.
    \item [(5)] For all finite places $v$ of $F$, $\rho|_{G_{F_v}}$ and $\rho_{\pi,\p}|_{G_{F_v}}$ are pure. 
\end{enumerate}

Then $\rho$ is twisted modular hence potentially modular (see \ref{twist}). More precisely, there exists a twisted ordinary automorphic representation $\pi'$ of $\GSp_4(\mathbb{A}_F)$ of parallel weight 2 that satisfies $\rho_{\pi',\p}\cong \rho$. 
\end{thm}

In \S 6, we showed that we can find an abelian variety $\mathcal{A}$ such that $\bar{\rho}_{\mathcal{A},\q}$ is induced from $\GL_2$ representations. The next theorem shows that such $\GL_2$ representations are potentially modular.

\begin{thm} (\cite{BCGP} Theorem 9.1.11)\label{modularinduced} Let $F_1/F$ be a finite extension of totally real fields and let $p,q>2$ be distinct primes that splits completely in $K$ and $F_1$. Fix a prime $\q\mid q$ in $K$. Let $\bar{r}:G_{F_1}\to \GL_2(\overline{K}_{(\q)})\cong\GL_2(\overline{\bbF}_q)$ be a representation with determinant $\bar{\epsilon}^{-1}$. 

Suppose that for each place $w\mid q$ of $F_1$, $\bar{r}|_{F_{1,w}}$ is of form $\Bigg(\begin{matrix}\bar{\lambda}_{\alpha_w}&0\\0&\bar{\epsilon}^{-1}\bar{\lambda}_{\alpha_w}^{-1}
\end{matrix}\Bigg)$, and suppose that $\bar{r}$ is unramified at all places above $p$. 

Then $r$ is potentially modular as $\GL_2$ representation. More precisely, let $F^{0}/F$ be a finite extension, then there is a finite Galois extension $F'/F$ of totally real fields in which $p$ and $q$ split completely and which is linearly disjoint from $F_1F
^0/F$, and a $\q$-ordinary cuspidal automorphic representation $\Pi$ of $\GL_2(\mathbb{A}_{F_1F'})$ of weight 0 and trivial central character which is unramified at all places dividing $pq$ and satisfies $\bar{\rho}_{\Pi,\q}\cong\bar{r}|_{G_{F_1F'}}$. 

\end{thm}

\begin{thm}\label{potmod}
Let $\mathcal{R}$ be a geometric compatible system of $\GSp_4$ representations satisfying the large image property and falls in one of the two cases: 
\begin{itemize}
    \item $\mathcal{R}$ is defined over $G_\bbQ$ and is strongly positive ordinary; In other words, $\mathcal{R}$ comes from a non-CM strongly positive ordinary abelian variety $A/F$ of $\GSp_4$-type that has descent data to $\bbQ$ and satisfies the large image property;
    \item $\mathcal{R}$ has coefficients in $\bbQ$ except for a set of primes of potentially density zero, which includes the case where $\mathcal{R}$ comes from an abelian surface $A$ with endomorphism ring $\bbZ$ or a fake abelian surface with endomorphism ring $\OO_D$ for a quaternion algebra $D/\bbQ$.
\end{itemize} Then $\mathcal{R}$ (hence $A$) is potentially modular. 
\end{thm}
\begin{proof}
We first apply Lemma \ref{suff1} in the first case and Lemma \ref{suff} in the second case to get the good primes $p$ and $q$. Following the notations of Theorem \ref{trick}, we know that there exists an abelian variety $\mathcal{A}/F'$ over a totally real extension field $F'$ in which $p,q$ splits completely, such that $\mathcal{A}[\p]\cong A[\p]$, and $\mathcal{A}[\q]\cong \bar{\rho}^\vee$ or $(\bar{\rho}^\vee)^2$, where $\bar{\rho}=\mathrm{Ind}_{G_{F'}}^{G_F}\bar{r}$ is induced with $\bar{r}$ satisfying properties in Lemma \ref{Induced}. 

By Theorem \ref{modularinduced}, after replacing $F'$ with a further totally real extension, we can maintain all assumptions in Theorem \ref{trick} and there exists a
$\q$-ordinary automorphic representation $\Pi$ of $\GL_2(\mathbb{A}_{F_1F'})$ of weight 0 and trivial central character and is unramified at all places dividing $pq$ that satisfies $\bar{\rho}_{\Pi,\q}\cong\bar{r}|_{G_{F_1F'}}$. 

It follows from \cite[Thm. 8.6]{Rob} that there is an automorphic representation $\pi$ of $\GSp_4(\mathbb{A}_{F'})$ of parallel weight 2 and trivial central character whose transfer to $\GL_4(\mathbb{A}_{F'})$ is the automorphic induction of $\Pi\otimes|\cdot|$. So $\rho_{\pi,\q}\cong\mathrm{Ind}_{G_{F_1F'}}^{G_{F'}}\rho_{\Pi,\q}$, and $\bar{\rho}_{\pi,\q}\cong\bar{\rho}|_{G_{F'}}$. In addition, $\pi$ is ordinary, and $\rho_{\pi,\lambda}$ is pure at all finite places since $\rho_{\Pi,\lambda}$ is.

Now we first apply Theorem \ref{modular} to $\rho_{\mathcal{A},\q}$ and conclude that it is modular. Thus $\rho_{\mathcal{A},\p}$ is modular, and we may apply Theorem \ref{modular} again to deduce that $\rho_{A,\p}|_{G_{F'}}$ is modular, thus $A$ is potentially modular. 
\end{proof}

\newpage
\section{$K3$ Surfaces with Large Picard Rank}

\subsection{Compatible System of Representations Associated to $K3$ Surfaces}
Let $F$ be a totally real field and $X/F$ be a $K3$ surface with geometric Picard number $17$. We are interested in considering the transcendental motive $T$ which is complement of the image of the cycle classes. Let $\mathcal{S}:=(\bbQ,S,\{P_{v}(X)\},\{\rho_p\},H_\tau)$ denote the compatible system of Galois representations $V_p:=T_p(-1)$ which are the $p$-adic realizations of $T(-1)$ inside $H^2(X,\bbQ_p)$.  There is a perfect orthogonal pairing
$$H^2(X,\bbQ_p(1)) \times H^2(X,\bbQ_p(1)) \rightarrow \bbQ_p$$
which gives rise to such a pairing on $T_p$. It follows that $\det(\rho_p) = \epsilon^{5}_p \eta_p$ where $\epsilon_p$ is the $p$-adic cyclotomic character  and $\eta_p = \det(\rho_p(-1))$ is also self-dual, and thus $\eta_p = \eta^{\vee}_p = \eta^{-1}_p$ is at most quadratic. The determinant is also a compatible system and thus is independent of $p$, and hence we write it simply as $\eta$.

We have the following well-known properties for the compatible system $\mathcal{S}$: 
\begin{enumerate}
\item[(1)] If $v \notin S$ and $v \nmid p$ then $\rho_p$ is unramified at $v$
and $\rho_p(\Frob_v)$ has characteristic polynomial $P_v(X)\in\bbZ[x]$.
\item[(2)] For all $p$,  the representation $\rho_p$ is de Rham with Hodge--Tate weights $H = [0,1,1,1,2]$.
If in addition $v \notin S$, then $\rho_p$ is crystalline,
and the characteristic polynomial of crystalline Frobenius $\Frob_v$ is $P_v(X)$.
\item[(3)] The representation $\rho_p$ is pure of weight $2$.
\item[(4)] \label{dJ} If $v \in S$ and $v \nmid p$, then the semi-simplification of the representation
to $\rho_p |_{I_{F_v}}$ has uniformly bounded order. That is, there exists a fixed finite extension of $K/F_v$
 such that $\rho_p |_{G_{K}}$ is unipotent on inertia.
\end{enumerate}

Our main goal would be to construct from $\mathcal{S}$ a compatible family $\mathcal{R}$ of $4$-dimensional
representations to $\GSp_4(\overline{\bbQ}_p)$ that lifts the representations $\rho_p \otimes \eta$ via the degree $2$ isogeny $\GSp_4 \rightarrow \mathrm{GO}_5$. 

\begin{thm}\label{prep} Suppose that there exists a prime $p$ for which $T_p$  has monodromy group with
Lie algebra containing $\mathfrak{so}_5$.
Then:
\begin{enumerate}
\item [(1)] \label{comp} There exists a simple abelian variety $A/F$ with $\End^0(A)=D$ having center $M$ and
a  weakly compatible system
$$\mathcal{R} = (M,S,\{Q_{v}(X)\},\{\tilde{\rho}_{\lambda}\},[0,0,1,1])$$ of Galois
representations such that, for all $l$ that totally splits in $M$ and all $\lambda\mid l$ in $M$,  the representation $\tilde{\rho}_{\lambda}$ is  irreducible, occurs inside $H^1(A,\overline{\bbQ}_l)$, and 
the composition
$$\tilde{\rho}_{\lambda}: G_{\bbQ} \rightarrow \GSp_4(\overline{\bbQ}_l) \rightarrow \mathrm{GO}_5(\overline{\bbQ}_l)$$
is isomorphic to $\rho_l \otimes \eta$.
\item [(2)] \label{fake} There exists a finite Galois extension $E/F$ with $\Gal(E/F) = (\bbZ/2\bbZ)^m$ such that the restriction of this compatible system comes
from either an abelian surface $A/E$ or a fake abelian surface $A/E$ with descent data to $F$. 
\end{enumerate}
\end{thm}

\begin{proof}
The irreducible representation $\rho_p \otimes \eta$ lifts to a representation $\tilde{\rho}: G_{F} \rightarrow \GSp_4(\overline{\bbQ}_p)$
with Hodge-Tate weights $[0,0,1,1]$ by \cite{Patrikis} \S 3.2. Furthermore, according to Proposition 4.2.31 of \cite{Patrikis}, $\tilde{\rho}$ lies in some abelian variety $A$ after some finite extension. Thus by Lemma \ref{extend}, it can be extended to some compatible system $\mathcal{R} = (M,S,\{Q_{v}(X)\},\{\tilde{\rho}_{\lambda}\},[0,0,1,1])$ that comes from the abelian variety $A$. 

Then we have $\Gal(M/\bbQ)$ acting on $\mathcal{R}$ via coefficients. 
If $\sigma \in \Gal(M/\bbQ)$, we see that $\mathcal{R}^{\sigma}$
is also a compatible system, and the corresponding compatible system of $\mathrm{GO}_5$ representations given by $\wedge^2 \mathcal{R}^{\sigma}$ is equal to the compatible system associated to $\rho_p \otimes \eta$.
We deduce that $\mathcal{R}^{\sigma}$ is isomorphic to $\mathcal{R}$ up to at most a quadratic twist by irreducibility.
Hence there exists
an extension $E/F$ (with Galois group $(\bbZ/2 \bbZ)^m$ for some $m$)
where all these twists become trivial. But then $\mathcal{R}$ restricted to $G_E$
is invariant under $\Gal(M/\bbQ)$, and hence forms a compatible system
with coefficients in $\bbQ$. But now the proof of \cite{BCGP} Lemma~10.3.2 implies that $\mathcal{R}$
over $E/F$ is associated to either an abelian surface or fake abelian surface $A$ over $E$,
and the fact that $\mathcal{R}$ can be extended to $G_{F}$ and gives the same Galois representations up to twist proves the descent data.

\end{proof}

We now show the irreducibility of one (hence for all) $T_p$ together with the large image property: 
\begin{lem}\label{noone}
    The representation $T_p$ can not have any one-dimensional factors.
\end{lem}

\begin{proof} Any one dimensional factor must correspond to a representation of $G_{F}$ with Hodge--Tate weight $-1$, $0$, or $1$. If it has Hodge-Tate weight $0$, then the character has finite order. By the Tate conjecture for $K3$ surfaces over a number field it must come from a ($\overline{\Q}_p$-linear combination of) cycle classes, which contradicts the assumption that $T$ is the transcendental motive. If the Hodge--Tate weight is $-1$ or $1$, it must come from $\chi \epsilon^{-1}$ or $\chi \epsilon$ where $\chi$ is some finite order character and $\epsilon$ is the cyclotomic character. But this contradicts purity.
\end{proof}

\begin{lem} \label{foroneforall} If $T_p$ is irreducible for one $p$, then it is absolutely irreducible for all $p$
and  the Lie algebra of the monodromy group
contains $\mathfrak{so}_5$.
\end{lem}

\begin{proof}
Suppose that $T_p$ is irreducible for one $p$. We first claim that $T_p$ must be absolutely irreducible. 
Note that $\dim T_p = 5$ is prime. Hence if it is reducible but not absolutely irreducible, then $T_p \otimes_{\Q_p} \overline{\bbQ}_p$
has to decompose into the product of $1$-dimensional representations which are conjugate by the action of $G_{\bbQ_p}$.
These one dimensional representations must be Hodge-Tate, and thus of the form $\chi\,\epsilon^m$ for some $m$
where $\chi$ is a finite order character and $\epsilon$ is the cyclotomic character. But $G_{\bbQ_p}$ fixes $\epsilon$,
which means that all the Hodge-Tate weights of $T_p$ are equal, a contradiction. 

Hence we may assume that $T_p$ is absolutely irreducible. Suppose that $T_p$ becomes reducible over a finite extension. Since $G_F$ acts transitively
on the factors they must all have the same dimension, which must therefore be equal $1$.  Moreover the dimension of each
isotypic component is independent of the character. The characters cannot all be the same since $T_p$
has distinct Hodge-Tate weights. Hence all the characters are distinct, and thus $T_p$ is induced from a character of a degree $5$ extension $K/F$ which is de Rham and hence an algebraic character by \cite[Lemma~4.3]{CG}. But such a $K$ cannot be CM and hence the only algebraic characters are finite order characters times characters of $G_{F}$, and the induction of such a character has all Hodge-Tate weights equal, which is again a contradiction. 
 
Thus the monodromy representation of $T_p$ is connected. The only proper Lie subalgebra of $\mathfrak{so}_5 \simeq \mathfrak{sp}_4$ with irreducible representations of dimension $5$ is ${\mathfrak{sl}}_2$ acting via the $4$th symmetric power representation. But if $T_p$
is the $4$th symmetric power of some representation over some finite extension of $G_{F}$, then the Hodge-Tate weights at any
embedding will necessarily be in arithmetic progression which they are not.

Hence $T_p$ must have monodromy group $\mathfrak{so}_5$.
But now by Theorem \ref{prep} there exists a compatible system $\mathcal{R}$ of absolutely irreducible $\GSp_4$
representations $\tilde{\rho}_{\lambda}$ and
it follows by constancy of ranks that the monodromy at each prime has semisimple part $\mathfrak{so}_5=\mathfrak{sp}_4$
or $\mathfrak{so}_4={\mathfrak{sl}}_2 \times {\mathfrak{sl}}_2$. In the latter case, the representation $T_l$ decomposes into a direct sum of a one-dimensional representation and a $4$ dimensional representation, which is not possible from the previous lemma. 
\end{proof}

It follows that if $T_p$ is not absolutely irreducible then it must decompose as $A_p \oplus B_p$  where $\dim(A_p)=2$ and $\dim(B_p) = 3$ are both
irreducible.

\begin{lem} \label{bounded} If $T_p$ is reducible, then $A_p$ is the Galois representation associated to a CM modular form of level only divisible by primes in $S$ and bounded independently of $p$.
\end{lem}

\begin{proof} Since $T_p$ carries a generalized orthogonal form,  it follows that $A_p$ and $B_p$ are both also orthogonal and in particular self-dual up to twist since they are irreducible with different dimensions.
Hence either $A_p$ has Hodge-Tate weights $[-1,1]$ and $B_p$ has Hodge--Tate weights $[0,0,0]$, or $A_p$ has Hodge-Tate weights $[0,0]$ and $B_p$ has Hodge-Tate weights $[-1,0,1]$. The Galois representation associated to $A_p$ has image in $\mathrm{GO}_2(\bbZ_p)$, thus is potentially abelian. 

The Fontaine-Mazur conjecture is known for potentially abelian representations \cite{Taguchi2003}. In particular, if $A_p$ has Hodge-Tate weights $[0,0]$, then it has finite image, which would contradict the Tate conjecture as in the proof of Lemma \ref{noone}.  Hence for any such $p$ we may assume that $A_p$ has Hodge-Tate weights $[-1,1]$ and $A_p(-1)$ is the Galois representation associated to a CM modular form of weight $3$ with good reduction outside $S$. The uniform boundedness of level at primes in $S$ is deduced from property (4) of the compatible system $\mathcal{S}$. 
\end{proof}

We now prove the desired irreducibility statement.

\begin{thm} The representation $T_p$ is absolutely irreducible for all $p$ with monodromy group containing $\mathfrak{sp}_4$.
\end{thm}

\begin{proof} By Lemma \ref{foroneforall}, it suffices to assume that $T_p$ is reducible for all primes $p$ and reach a contradiction.
If $T_p$ is reducible, then by the previous lemma the 3-dimensional factor $B_p$ has Hodge-Tate weights $[0,0,0]$, and for each $p$, the 2-dimensional Galois representation $A_p(-1)$ comes from one of only finitely many weight $3$ CM modular forms $\pi$. All such forms are induced from a finite set of imaginary quadratic fields (those unramified outside $S$).

Hence we may assume that, for some set of primes $\Sigma$ of strongly positive density that split completely in each of these quadratic fields, that $A_p$ is associated to a fixed such $\pi$, and in particular $A_p = \mathrm{Ind}^{G_F}_{G_E} \chi_p$ for a fixed imaginary quadratic field $E/F$ and Gr\"ossencharacter $\chi_p$ of $E$. 

Let $P_{A_p,v}(X)$ denote the characteristic polynomial of Frobenius on $A_p(-1)$ at $v$ for $v \notin S \cup \{p\}^F$, and let $P_{A_p,v}(X)$ denote the characteristic polynomial of crystalline Frobenius on $A_p(-1)$ at $v$ assuming that $v \mid p$ and $v \notin S$. Because the Gr\"ossencharacters give rise to a weakly compatible system, these polynomials do not depend on $p$ as long as $p \in \Sigma$. Hence we may write them as $P_{A,v}(X)$, and similarly we may write $P_{B,v}(X)$. Since $A_p(-1)$ occurs in $H^2(X,\Q_p)$ we also see that these polynomials have integral coefficients. 

We now have a factorization
$$P_v(X) = P_{A,v}(X) P_{B,v}(X).$$
for all $v$. Recall that the Hodge-Tate weights of $\rho_p$ are $[0,1,1,1,2]$. On the other hand, since primes $p \in \Sigma$ split in $E$ by construction, the representation $A_p$ is ordinary at $p$ and hence the Newton Polygon of $P_{A,v}(X)$ has slopes $0$ and $2$. By Newton over Hodge, this implies that the Newton Polygon of $B_p$ has all slopes equal to $1$. That means that the roots of $P_{B,v}(X)$
are all divisible by $p$ as algebraic integers. But by purity they also have absolute value $p$ for all complex embeddings, and thus it follows that the roots of $P_{B,v}(pX)$ are roots of unity. Since $P_v(X)$ is a degree $5$ polynomial over $\Q$, they are indeed roots of unity of bounded order. Since there are only finitely many such roots of unity, we deduce that for a set $\Sigma$ of strongly positive density, the polynomial $P_{B,v}(X)$ is given by a fixed polynomial. On the other hand, up to twist, the Galois representation associated to $B_p$ has image inside $\mathrm{SO}_3(\Q_p)$. Hence either the monodromy group of $B_p$ is connected, in which case $B_p$ would have Lie algebra ${\mathfrak{sl}}_2$ acting via the symmetric square, or $B_p$ is potentially abelian with finite image or induced.

In the latter case we can again apply the Fontaine-Mazur conjecture to get a contradiction, so we may assume that the monodromy group of $B_p$ has Lie algebra $\mathfrak{sl}_2$. But then from the Cebotarev density theorem we see that any subset of primes with a fixed characteristic polynomial has density zero, a contradiction. 

\end{proof}

In order to apply our potential modularity results in Chapter 6 we need an extra lemma on ordinariness: 
\begin{lem} \label{ordinary} For all primes $p$ except for a potentially density zero set of primes, the representation $T_p(-1)$ contains at least one crystalline eigenvalue which is a $p$-adic unit. Moreover, the representations $\rho_{\lambda}$ are also ordinary for all $\lambda \mid p$  with distinct unit eigenvalues mod $\lambda$. 
\end{lem}
  
\begin{proof} 
Note that $T_p(-1)$ has Hodge-Tate weights $[0,1,1,1,2]$. If all crystalline eigenvalues were not $p$-adic units, then the polynomial $P_v(X) = X^5 + t_v X^4 + \ldots + p^5 \eta(p)$ would have $p \mid t_v$. By purity, we have $|t_v| \le 5p$. Thus we would have to have $t_v/p = a$ for some fixed integer $|a| \le 5$, and some set of primes $v$ of strongly positive density. Arguing
as in the proof of Lemma \ref{nolinear}, this contradicts the fact that the monodromy group of $T_p$ contains $\mathfrak{so}_5$ for any prime $p$. This proves the first part of the lemma.
  
Now suppose $\tilde{\rho}_{\lambda}: G_{F} \rightarrow \GSp_4(\bar{\Q}_p)$ be the representation lifting $T_p(-1)$. Note that all of the crystalline Frobenius eigenvalues of $\tilde{\rho}_\lambda$ will be algebraic integers. Suppose they are $\alpha_v, \beta_v, p\zeta\alpha_v^{-1}, p\zeta\beta_v^{-1}$,  where $\zeta$ is a root of unity of uniformly bounded order.
Up to symmetry, if $T_p(-1)$ has a crystalline eigenvalue which is an $p$-adic unit then we may assume that $\alpha_v \beta_v$ is prime to $p$, and hence that $\rho_{\lambda}$ is ordinary for all $\lambda \mid p$ and all $p$ except for a potentially density zero set.

Write \begin{align*}
    Q_{v}(X) &= X^4 - a_v X^3 + b_v X^2 - p\zeta a_v X + p^2\zeta^2
\\&= (X - \alpha_v)(X - \beta_v)(X - p\zeta\alpha_v^{-1})(X - p\zeta\beta_v^{-1}) \in \OO_M[X].
\end{align*}
The ordinary assumption implies that $b_v \not\equiv 0 \bmod \lambda$ for all $\lambda \mid p$.
Suppose that the unit root eigenvalues are the same for some $\lambda \mid p$. Then $a^2_{v} - 4 b_{v} \equiv 0 \bmod \lambda$.
On the other hand if $\sigma \in \Gal(M/\Q)$, then $a^{\sigma}_v = \chi(\sigma) a_v$ and $b^{\sigma}_v = \chi(\sigma)^2 b_v$
for some finite order character $\chi$, and so $(a^{\sigma}_v)^2 - 4 b^{\sigma}_v = \chi(\sigma)^2 (a^2_v - 4 b_v)$,
hence if this is divisible by one $\lambda$ then it is divisible by all, and thus (at least for $p$ unramified in $M$)
that $a^2_v - 4 b_v$ is divisible by $p$. Since $\rho_{\lambda}$ is pure of weight one, it follows that $(a^2_v - 4 b_v)/p \in \OO_M$ has absolute value bounded uniformly in $p$, and there are only finitely many such elements in $\OO_M$ with this property.

But once more arguing as in the proof of Lemma \ref{nolinear},
the fact that the image of $\tilde{\rho}_{\lambda}$ has monodromy group containing $\mathfrak{sp}_4$ for all $\lambda$ shows that the unit eigenvalues would be different modulo $\lambda$ for a set of primes $p$ except for a potentially density zero set. 
\end{proof} 

Now that we have constructed the abelian variety $A/E$ with descent data to $F$ that gives the compatible system $\mathcal{R}$ and proved that it is strongly positive ordinary, thus we have:
\begin{cor} The compatible system of $\GSp_4$ representation $\mathcal{R}$ constructed in Theorem \ref{prep} that corresponds to the transcendental motive $T$ is potentially modular. Therefore, Theorem \ref{main} is true when $X$ has Picard rank 17. 
\end{cor}

\begin{proof}
The compatible system defined in Theorem \ref{prep} falls into the second case of Theorem \ref{potmod}, where we consider $\mathcal{R}$ to be the compatible system of representations of $G_E$ with coefficients in $\bbQ$ except for a set of primes of potentially density zero. In fact, Lemma \ref{ordinary} shows that we can find good primes $p$ and $q$ to perform the trick in Theorem \ref{trick} and thus the desired result just follows from Theorem \ref{potmod}.   
\end{proof}

\subsection{The Geometric Picard Rank $\ge 18$}

Now suppose $X/F$ has geometric Picard rank $18$. Then associated to $T_p(-1)$ is a compatible system $\mathcal{S}=(M,S,\{Q_v(X)\},\{\rho_p\},[0,1,1,2])$. We now derive similar results for $T_p$: 

\begin{lem}
The semisimple part of the Lie algebra of the monodromy group of $T_p$ is $\mathfrak{so}_4 \simeq  {\mathfrak{sl}}_2 \oplus {\mathfrak{sl}}_2$.
\end{lem}

\begin{proof} Suppose that $T_p$ is not absolutely irreducible. 
Since Lemma 8.2 also applies here, it follows that $T_p \otimes \overline{\bbQ}_p$ it is a direct sum of two factors $A_p$ and $B_p$ of dimension $2$. We show that this cannot occur even over some finite extension of $F$. If $A_p$ and $B_p$ are both orthogonal, they are both induced, and then $T_p$ is potentially abelian.
Thus without loss of generality $A_p$ is not orthogonal, which implies that the dual of $A_p$ under the orthogonal pairing is a twist of $B_p$. Since two-dimensional representations are self-dual up to twist, it follows that $T_p \simeq A_p \oplus A_p \otimes \chi$ for some character $\chi$ which will be algebraic. 

But since $T_p$ is pure of weight $1$, it follows that $A_p$ is also pure of weight one, and that $\chi$ is pure of weight zero, and thus has finite image. But then the Hodge--Tate weights of $T_p$ must each have multiplicity two, which is a contradiction.

Suppose that $T_p$ is absolutely irreducible but becomes reducible after a finite extension. The factors must each have the same dimension. If they have dimension one then $T_p$ is potentially abelian. If they have dimension $2$ then $T_p$ is once more of the form $A_p \oplus B_p$ after restricting to some finite extension which we have already considered. Thus $T_p$ remains irreducible over any finite extension and the only possibility is that the monodromy group is $\mathfrak{so}_4$.
\end{proof}

There is an isogeny $\GL_2(\overline{\bbQ}_p) \times \GL_2(\overline{\bbQ}_p) \rightarrow \mathrm{GO}_4(\overline{\bbQ}_p)$ whose image is the connected subgroup of index two. By an argument of Serre \cite{serre1981letter}, the component group of a compatible family is independent of $p$, and thus there exists an  (at most) degree two extension $L/F$ such that the image of $T_p$ lands inside this image.

\begin{thm}
\label{preptwo} Suppose that there exists a prime $p$ for which $T_p$  has monodromy group with
Lie algebra $\mathfrak{so}_4$.
Then there exists an extension $L/F$ of degree at most two over which the monodromy of $T_p$ is connected such that:
\begin{enumerate}
\item[(1)] \label{comptwp} There exists  simple abelian varieties $E_1/L$ and $E_2/L$ and
weakly compatible families
$$\mathcal{R}_i = (M,S,\{Q_{v,i}(X)\},\{\tilde{\rho}_{\lambda,i}\},[0,1])$$
 of Galois
representations such that, for all $l$ and all $\lambda\mid l$,  the representation $\tilde{\rho}_{\lambda,i}$ is  irreducible, occurs inside $H^1(E_i,\overline{\bbQ}_l)$
as a $\Gal(\overline{\bbQ}/L)$ representation, and 
the composite
$$\tilde{\rho}_{\lambda}: G_{L} \rightarrow \GL_2(\overline{\bbQ}_p) \times \GL_2(\overline{\bbQ}_p) \rightarrow \mathrm{GO}_4(\overline{\bbQ}_l)$$
is isomorphic to $\rho_l $ restricted to $G_L$. 
\item[(2)]  \label{faketwo} There exists a finite Galois extension $H/L$ such that the restriction of these compatible families comes 
from either  elliptic curves $E_i/H$ or  fake elliptic curves $E_i/H$ with descent data to $L$, that is, there exist isogenies between $E_i$ and $E^{\sigma}_i$
for any $\sigma \in \Gal(H/L)$.
\end{enumerate}
\end{thm}

\begin{proof}
By \cite[\S3.2]{Patrikis}, the representation $\rho_p: G_{E} \rightarrow \mathrm{GO}_4(\Q_p)$ lifts to a representation $\tilde{\rho}: G_{E} \rightarrow \GL_2(\overline{\bbQ}_p) \times \GL_2(\overline{\bbQ}_p)$ with Hodge-Tate weights $[0,0,1,1]$ and monodromy group whose Lie algebra has semisimple part $\mathfrak{sl}_2\oplus \mathfrak{sl}_2$.
By \cite[Lemma 4.2.22]{Patrikis}, over some finite extension $E'/E$, the representation $\tilde{\rho}$ is part of a compatible family of Galois representations associated to (part of) some abelian variety. In particular, the $2$-dimensional constituents also come from abelian varieties $E_i/L$ giving rise to $\mathcal{R}_1$ and $\mathcal{R}_2$ respectively, with coefficients jointly in some finite extension $F$. Now consider the action of $\Gal(M/\Q)$ on the coefficients, we have
$$\mathcal{R}^{\sigma}_1 \otimes \mathcal{R}^{\sigma}_2 \simeq \mathcal{R}_1 \otimes \mathcal{R}_2 \simeq \mathcal{S},$$
because the latter has coefficients over $\Q$. 
If follows that $\mathcal{R}^{\sigma}_i = \mathcal{R}_{j} \otimes \chi_{\sigma}$ for each $\sigma$ where $j$ may or may not be equal to $i$,
and $\chi_{\sigma}$ is a finite order character. If there exists a $\sigma$ such that $\mathcal{R}^{\sigma}_1$ is a twist  of $\mathcal{R}_2$,
then their monodromy groups are the same on a finite index, which contradicts the fact that the relevant Lie algebra is $\mathfrak{sl}_2 \oplus \mathfrak{sl}_2$
and not $\mathfrak{sl}_2$. Hence $\mathcal{R}^{\sigma}_i$ is a twist of $\mathcal{R}_i$ for each $\sigma$. Letting $H/L$ denote the fixed field of these characters,
we deduce that $\mathcal{R}_i$ as a compatible family  over $G_L$ has coefficients in $\Q$. But now we are done
using the same idea of the proof of \cite[Theorem~10.3.2]{BCGP}. 
\end{proof}

\begin{thm}
Theorem \ref{main} is also true when $X$ has Picard rank $\ge 18$. 
\end{thm}
\begin{proof} Suppose the Picard rank is $18$. Then it suffices to show that the compatible families $\mathcal{R}_1$ and $\mathcal{R}_2$
constructed in Theorem \ref{preptwo} are potentially modular, and then to use the automorphic of tensor products $\GL_2 \times \GL_2 \rightarrow \GL_4$ \cite{Ram}.
But the potential (simultaneous) modularity of $\mathcal{R}_1$ and $\mathcal{R}_2$ follows from \cite[Theorem 7.1.10]{ACG}.

If the Picard rank is $19$ or $20$, then $\mathcal{S}$ satisfies regularity (distinct Hodge-Tate weights) and thus potentially automorphic by \cite{blggt} Theorem A.
\end{proof}

\begin{obs}
    When the Picard rank is $19$ or $20$, it is easy to deduce the stronger claim that $\mathcal{S}$ is automorphic. In the rank $20$ case the
representation is orthogonal and thus induced. In the rank $19$ case, the compatible system of $\mathrm{GO}_3$ representations lifts as in the proof of Theorem \ref{prep}
to a $2$-dimensional family associated with a $\GL_2$-type abelian variety with descent data to $F$, and hence the result
follows from \cite{Ribet} and the proof of Serre's conjecture \cite{SerreKW1,SerreKW2}.
\end{obs}
\newpage
\section*{Declarations}

\subsection*{Data Availability Statement}
No datasets were generated or analyzed during the current study. This research is purely theoretical and did not involve the collection or use of any empirical data.

\subsection*{Conflict of Interest Statement}
The authors declare that they have no known competing financial interests or personal relationships that could have appeared to influence the work reported in this paper.
\newpage
\bibliography{1.bib}

\end{document}